%% file: WeylGroupoids.tex
\documentclass[a4paper,reqno]{amsart}

\input{WeylGroupoidspreamble.tex}

\title{The Weyl groupoids of $\lie{sl}(m|n)$ and $\lie{osp}(r|2n)$}

\author{Lukas Bonfert}
\address{L.B.: Max Planck Institute for Mathematics, Bonn, Germany}
\author{Jonas Nehme}
\address{J.N.: Max Planck Institute for Mathematics, Bonn, Germany}

\date{\today}


\begin{document}
	\begin{abstract}
		We provide a convenient formulation of the definition of Cartan graphs and Weyl groupoids introduced by Heckenberger and Schneider, and construct Cartan graphs for regular symmetrizable contragredient Lie superalgebras.
		For $\lie{sl}(m|n)$, $\lie{osp}(2m+1|2n)$ and $\lie{osp}(2m|2n)$ we explicitly describe the Cartan graph in terms of partitions and determine the relations between the generators of their Weyl groupoids.
	\end{abstract}
	\maketitle
	\section{Introduction}
		Let $\lie{g}$ be a basic classical simple Lie superalgebra and let $\lie{h}\subseteq\lie{g}$ be a Cartan subalgebra.
		It is well-known that, in contrast to the situation for semisimple Lie algebras, not all Borel subalgebras of $\lie{g}$ containing $\lie{h}$ are conjugate to each other.
		As a consequence there are several systems of simple roots that are not conjugate under the action of the Weyl group.
		The number of conjugacy classes is however finite (see e.g.~\cite[Thm.~3.1.2]{Musson}), which is equivalent to saying that there are only finitely many Borel subalgebras with fixed even part $\lie{b}_\even$.

		In \cite{PenkovSerganovaClassicalCohomology} Penkov and Serganova introduced \emph{odd reflections} to pass between Borel subalgebras with the same even part.
		Explicitly, they can be described as follows (see e.g.~\cite[\S 1.4]{ChengWang} or \cite[\S 3.5]{Musson}).
		Let $\{\alpha_1,\dots,\alpha_n\}$ be the simple roots corresponding to some Borel subalgebra $\lie{b}$ and suppose the simple root $\alpha_i$ is odd isotropic.
		Then the simple roots $\alpha'_j=r_i(\alpha_j)$ for the Borel subalgebra $\lie{b}'$ obtained from $\lie{b}$ by odd reflection at $\alpha_i$ are
		\begin{equation}
			\label{oddreflectionformula}
			\alpha'_j=r_i(\alpha_j)=\begin{cases}
				-\alpha_i&\text{if }j=i,\\
				\alpha_j&\text{if }j\neq i, \alpha_j(h_i)=0\\
				\alpha_j+\alpha_i&\text{if }j\neq i, \alpha_j(h_i)\neq 0,
			\end{cases}
		\end{equation}
		where $h_i\in\lie{h}$ is the coroot corresponding to $\alpha_i$.
		More generally, these definitions also work for contragredient Lie superalgebras.
		In \cite{SerganovaSuperKacMoody} the odd reflections (and certain other maps) are used to construct a \emph{Weyl groupoid} that acts transitively on the set of Borel subalgebras.
		
		On the other hand, a (seemingly unrelated at first glance) notion of Weyl groupoid was also introduced by Heckenberger and Yamane \cite{HeckenbergerYamane} as an analogue of Weyl groups in the theory of Nichols algebras.
		Weyl groupoids in this context are constructed from (semi-)Cartan graphs, see \cite[\S 9]{HeckenbergerSchneider}.
		A (semi-)Cartan graph is an undirected graph $\cartangr$ with edges labelled by $\{1,\dots,n\}$ and a generalized Cartan matrix $A(x)$ (called the Serre matrix) for every vertex $x$, subject to certain conditions.
		In \cite[Ex.~3]{HeckenbergerYamane} examples of Cartan graphs are obtained from finite-dimensional contragredient Lie superalgebras.
		Furthermore, \cite[\S 11]{HeckenbergerSchneider} provides a different combinatorial construction of Cartan graphs for regular contragredient Lie superalgebras, using only the Cartan data.
		Roughly speaking, the vertices of the Cartan graph are the ordered bases of the roots of $\lie{g}$, the edges correspond to odd reflections, and the Serre matrices define the Serre relations in $\lie{g}$.
		By results of \cite{HeckenbergerYamane} the Weyl groupoid of a Cartan graph is a Coxeter groupoid, i.e.~it is generated by \emph{simple reflections} $(s_i)_x\colon x\to r_i(x)$ subject only to Coxeter-type relations $\id_x(s_is_j)^{m(x)_{ij}}\id_x=\id_x$ (for any possible composition).
		
		Our first main result is a formulation of the general construction of Cartan graphs and Weyl groupoids from \cite{HeckenbergerSchneider} in more convenient graphical language, see \cref{weylcartandefinitions}.
		Moreover, in \cref{contragredientcartan} we generalize the construction of Cartan graphs for finite-dimensional contragredient Lie superalgebras from \cite{HeckenbergerYamane} to regular symmetrizable contragredient Lie superalgebras, and show that this is equivalent to the combinatorial definition from \cite{HeckenbergerSchneider}.
		In particular, this implies that we actually obtained a Cartan graph.
		In comparison to \cite{HeckenbergerSchneider} we put the focus on the Borel subalgebras instead of the Cartan data, which is advantageous from the perspective of Lie theory.
		This point of view allows us to compare this notion of Weyl groupoid to other constructions in the theory of Lie superalgebras, see \cref{groupoidcomparison}.
		In \cref{weylgroupoidautos} we show that the automorphism group of an object of the Weyl groupoid $\weylgrpd$ of a contragredient Lie superalgebra coincides with its Weyl group.
		
		Our second result is an explicit description of the Cartan graphs and Weyl groupoids for the Lie superalgebras $\lie{sl}(m|n)$, $\lie{osp}(2m+1|2n)$ and $\lie{osp}(2m|2n)$.
		These Weyl groupoids were previously considered in \cite{AndruskiewitschAngionoSurvey}.
		We provide a detailed, different description in terms of partitions.
		For this we first recall some standard results about the realizations of $\lie{sl}(m|n)$, $\lie{osp}(2m+1|2n)$ and $\lie{osp}(2m|2n)$ as contragredient Lie superalgebras, which amounts to classifying all Borel subalgebras up to conjugation.
		Based on \cite[\S 3]{Musson} we describe the Borel subalgebras (up to conjugation) in terms of partitions $\lambda$ fitting in an $m\times n$-rectangle.
		Given such a partition (and an additional sign $\eps\in\{+,-\}$ in the case of $\lie{osp}(2m|2n)$) one can easily construct a Borel subalgebra $\lie{b}(\lambda)$ (resp.~$\lie{b}(\lambda,\eps)$), see \cref{constructionBorelsl,constructionBoreloddosp,constructionBorelevenosp} for details.
		To determine the Cartan graphs we also need an explicit description of the the Cartan data for all Borel subalgebras, which we compute in \cref{cartancomputation}.
		
		An observation crucial to determining the Cartan graphs of the Lie superalgebras $\lie{sl}(m|n)$, $\lie{osp}(2m+1|2n)$ and $\lie{osp}(2m|2n)$ is that the combinatorial description of Borel subalgebras in terms of partitions allows for a convenient description of the odd reflections:
		an odd reflection corresponds to adding or removing a certain box to (resp.~from) $\lambda$, see \cref{oddreflectionpartition} for details.
		This makes it very easy to describe the shape of the Cartan graph in concrete examples, see \cref{cartangraphshape}.
		In \cref{cartanmatrices} we compute the Serre matrices, which are obtained from the Cartan data computed in \cref{borelsandcartan}.
		Finally we determine the Coxeter relations in their Weyl groupoids.
		Somewhat surprisingly, these are as one would expect from the Serre matrices, see \cref{coxeterrelations}.
		\subsection*{Acknowledgements}
			We would like to thank Catharina Stroppel for extensive discussions and valuable feedback on earlier drafts.
			This work was supported by the Max Planck Institute of Mathematics (IMPRS Moduli Spaces) and the Hausdorff Center of Mathematics, which is funded by the Deutsche Forschungsgemeinschaft (DFG, German Research Foundation) under Germany's Excellence Strategy -- EXC-2047/1 -- 390685813.
			We would like to thank both institutions for their hospitality and the excellent research environment.
	\section{Cartan graphs and Weyl groupoids}
		\subsection{Definition and generalities}\label{weylcartandefinitions}
			We begin by reformulating the definitions of Cartan graphs and Weyl groupoids from \cite[\S 9]{HeckenbergerSchneider}.
			The notion of Weyl groupoids was first axiomatically introduced in \cite{HeckenbergerYamane}.
			
			Let $I$ and $J$ be sets with $|I|<\infty$.
			By an \emph{$(I,J)$-labelled graph} we mean an (undirected) graph $\cartangr$ with vertices $X$ and edges $E$ together with maps of sets $A\colon X\to J$ and $c\colon E\to I$.
			For an edge $e$ we call $c(e)\in I$ its \emph{color}.
			We draw the $i$-colored edges of an $(I,J)$-labelled graph as $x\overset{i}{\longleftrightarrow}x'$.
			The set $I$ will usually be left implicit.
			
			For a finite set $I$ let $\GCM_I(\ints)$ be the set of generalized Cartan matrices with entries indexed by $I$.
			\begin{dfn}
				A \emph{semi-Cartan graph} is an $(I,\GCM_I(\ints))$-labelled graph such that
				\begin{enumerate}[label=(CG\arabic*),leftmargin=*,series=cartangraph]
					\item\label{cg1} for every vertex $x\in X$ and every $i\in I$ there is a unique edge $e$ incident to $x$ with $c(e)=i$,
					\item\label{cg2} and $A(x)_{ij}=A(y)_{ij}$ for every edge $x\overset{i}{\longleftrightarrow}y$ and all $j\in I$.
				\end{enumerate}
				The matrices $A(x)$ are called \emph{Serre matrices}.
			\end{dfn}
			We call the matrices $A(x)$ Serre matrices since in special cases they describe Serre relations, see \cref{serrematrixexplanation}.
			
			Note that loops in semi-Cartan graphs are explicitly allowed, and in fact occur very often.
			\begin{rem}
				By \cref{cg1} we obtain involutions $r_i\colon X\to X$, sending a vertex $x$ to its neighbor along the unique $i$-colored edge at $x$.
				This recovers the axioms in \cite[Def.~9.1.1]{HeckenbergerSchneider}.
			\end{rem}
			Let $\cartangr$ be a semi-Cartan graph.
			To each vertex $x\in X$ we associate a $\ints$-lattice $\ints^I_x$ with basis $\{\alpha_i^x\mid i\in I\}$, and for $i\in I$ we define $\ints$-linear maps $s_i=(s_i)_x\colon \ints^I_x\to\ints^I_{r_i(x)}$ by $(s_i)_x(\alpha_j^x)=\alpha_j^{r_i(x)}-A(x)_{ij}\alpha_i^{r_i(x)}$.
			\begin{dfn}\label{weylgroupoiddef}
				The \emph{Weyl groupoid} $\weylgrpd$ of $\cartangr$ is the groupoid with set of objects $X$ and the morphisms generated by the $(s_i)_x\colon x\to r_i(x)$ for $i\in I$, $x\in X$.
				The composition of morphisms is given by the usual composition of $\ints$-linear maps, and two morphisms $x\to y$ are equal if and only if they agree as $\ints$-linear maps $\ints^I_x\to\ints^I_y$.
			\end{dfn}
			Observe that from \cref{cg2} we get $(s_i)_{r_i(x)}(s_i)_x=\id_x$ for all $i\in I$ and $x\in X$, and thus $\weylgrpd$ is indeed a groupoid.
			\begin{rem}\label{importanceoflabelatbasis}
				Usually one thinks of the lattices $\ints^I_x$ as lying inside a fixed ambient $\compl$-vector space of dimension $|I|$.
				Obviously the lattices $\ints^I_x$ are isomorphic  as abstract lattices, but they are usually rather different (and in particular depend on the vertex $x$) as sublattices of the ambient vector space.
				To emphasize this we will always keep track of the vertex $x$ for the lattice $\ints^I_x$ and its basis vectors $\alpha^x_i$.
			\end{rem}
			\begin{dfn}\label{definitionrealroots}
				Let $\cartangr$ be a semi-Cartan graph and $\weylgrpd$ its Weyl groupoid, and define the following sets of roots at vertex $x\in X$:
				\begin{itemize}
					\item The \emph{real roots} are
						\begin{equation*}
							\cartanroots^\real_x=\{w(\alpha_i^y)\mid y\in X, w\in\weylgrpd(y,x), i\in I\}\subseteq\ints^I_x,
						\end{equation*}
					\item the \emph{positive} (resp.~\emph{negative}) real roots are $\cartanroots^{\pm,\real}_x=\cartanroots^\real_x\cap\pm\sum_{i\in I}\nats_0\alpha_i^x$.
				\end{itemize}
			\end{dfn}
			\begin{dfn}
				A semi-Cartan graph $\cartangr$ is a \emph{Cartan graph} if
				\begin{enumerate}[resume*=cartangraph]
					\item $\cartanroots^\real_x=\cartanroots^{+,\real}_x\cup\cartanroots^{-,\real}_x$ for all $x\in X$,
					\item\label{CG4} and for all $x\in X$ and $i,j\in I$ we have
						\begin{equation*}
							(r_ir_j)^{m(x)_{ij}}(x)=x,
						\end{equation*}
						where $m(x)_{ij}=\abs{\cartanroots^\real_x\cap(\nats_0\alpha_i^x+\nats_0\alpha_j^x)}$.
				\end{enumerate}
			\end{dfn}
		\subsection{Cartan graphs for contragredient Lie superalgebras}\label{contragredientcartan}
			Now we construct a generalized Cartan graph from a contragredient Lie superalgebra.
			We begin by recalling the construction of contragredient Lie superalgebras from \cite[\S 5]{Musson}.
			A \emph{Cartan datum} is a pair $(B,\tau)$ consisting of a matrix $B\in\compl^{n\times n}$ and a \emph{parity vector} $\tau\in (\ints/2\ints)^n$ for some $n\in\nats$.
			We further fix a \emph{minimal realization} of $B$, i.e.\ we choose a vector space $\lie{h}$ of dimension $2n-\rk(B)$ with linearly independent roots $\alpha_i\in\lie{h}^*$ and coroots $h_i\in\lie{h}$ for $1\leq i\leq n$ such that $\alpha_j(h_i)=a_{ij}$.
			This data can be used to construct a Lie superalgebra $\tilde{\lie{g}}(B,\tau)$ with \emph{Chevalley generators} $e_i$ and $f_i$ (of parity $\tau_i$) subject to the following relations which are analogous to the defining relations of Kac--Moody Lie algebras:
			\begin{gather*}
				[h,h']=0\quad\text{for all }h,h'\in\lie{h},\\
				[h, e_i]=\alpha_i(h)e_i\quad\text{for all }h\in\lie{h},\\
				[h, f_i]=-\alpha_i(h)f_i\quad\text{for all }h\in\lie{h},\\
				[e_i, f_j]=\delta_{i,j}h_i.
			\end{gather*}
			As usual, we call $\lie{h}$ the \emph{Cartan subalgebra}.
			Note that $\lie{h}$ is concentrated in even degree, and is abelian.
			Let $\lie{g}(B,\tau)=\tilde{\lie{g}}(B,\tau)/\lie{r}$, where $\lie{r}$ is the maximal ideal of $\tilde{\lie{g}}$ intersecting $\lie{h}$ trivially.
			From the construction of $\lie{g}(B,\tau)$ it is clear that rescaling the rows of $B$ by non-zero scalars results in isomorphic Lie superalgebras.

			In the following we will restrict ourselves to regular symmetrizable Cartan data in the sense of \cite[Def.~4.8]{HoytSerganova} (up to rescaling of rows).
			A Cartan datum $(B,\tau)$ (with the same notation as in \cref{borelsandcartan}) is called \emph{symmetrizable} if the matrix $B$ is symmetrizable.
			We call $(B,\tau)$ \emph{regular} if
			\begin{itemize}
				\item $B$ has no zero rows and is indecomposable (i.e.~does not split into blocks $B=\begin{smallpmatrix}B_1&0\\0&B_2\end{smallpmatrix}$),
				\item $b_{ij}=0$ if and only if $b_{ji}=0$,
				\item if $\tau_i=\even$ then $b_{ii}\neq 0$ and $\frac{2b_{ij}}{b_{ii}}\in\ints_{\leq 0}$ for all $j$, and
				\item if $\tau_i=\odd$ and $b_{ii}\neq 0$, then $\frac{2b_{ij}}{b_{ii}}\in 2\ints_{\leq 0}$ for all $j$.
			\end{itemize}
			Regularity of $(B,\tau)$ implies that $\ad_{e_i}\colon \lie{g}(B,\tau)\to\lie{g}(B,\tau)$ is locally nilpotent for all $i$, see \cite[\S 2]{SerganovaSuperKacMoody}.
			
			We call the Lie superalgebra $\lie{g}(B,\tau)$ \emph{regular} if any Cartan datum $(B',\tau')$ with $\lie{g}(B',\tau')\cong\lie{g}(B,\tau)$ is regular.
			In particular $\lie{sl}(m|n)$, $\lie{osp}(2m+1|2n)$ and $\lie{osp}(2m|2n)$ are regular, which follows from the computation of all possible Cartan data (up to rescaling of rows) in \cref{cartancomputation}.
			
			The following definition is \cite[Def.~11.2.4]{HeckenbergerSchneider}.
			It simplifies slightly since we have restricted ourselves to regular Cartan data.
			\begin{dfn}
				\label{associatedCMdef}
				For a regular Cartan datum $(B,\tau)$, the \emph{Serre matrix} $A^{B,\tau}$ is the $n\times n$-matrix with entries
				\begin{equation*}
					a^{B,\tau}_{ij}=\begin{cases}
						2&\text{if }i=j,\\
						0&\text{if }i\neq j, b_{ij}=0,\\
						-1&\text{if }i\neq j, b_{ij}\neq 0, b_{ii}=0,\\
						\frac{2b_{ij}}{b_{ii}}&\text{if }i\neq j, b_{ij}\neq 0, b_{ii}\neq 0.
					\end{cases}
				\end{equation*}
			\end{dfn}
			Observe that due to regularity of $(B,\tau)$ the Serre matrix $A^{B,\tau}$ is a generalized Cartan matrix.
			Also note that in particular $A^{B,\tau}$ is invariant under multiplication of rows of $B$ by non-zero scalars.
			\begin{rem}
				The definition of $A^{B,\tau}$ can be seen as a normalization of $B$, bringing it as close as possible to the form of a generalized Cartan matrix.
				The rows with non-zero diagonal entries are rescaled so that the diagonal entries become $2$, while for odd isotropic roots (those with $b_{ii}=0$) we replace all non-zero off-diagonal entries by $-1$.
				In particular if $B$ is a generalized Cartan matrix, then $A^{B,\tau}=B$.
			\end{rem}
			\begin{rem}
				\label{serrematrixexplanation}
				We call the matrix $A^{B,\tau}$ a Serre matrix since it is used to formulate Serre relations for the contragredient Lie superalgebra $\lie{g}(B,\tau)$.
				Explicitly we have
				\begin{equation*}
					(\ad e_i)^{1-a_{ij}}(e_j)=0,
				\end{equation*}
				see for instance \cite[Lem.~5.2.13]{Musson}.
			\end{rem}
			For the rest of the section fix a regular symmetrizable contragredient Lie superalgebra $\lie{g}=\lie{g}(B,\tau)$ with $B\in\compl^{n\times n}$ and let $I=\{1,\dots,n\}$.
			
			As a direct consequence of the construction, $\lie{g}$ admits a root space decomposition $\lie{g}=\lie{h}\oplus\bigoplus_{\alpha\in\lie{h}^*} \lie{g}_\alpha$.
			An \emph{ordered root base} of $\lie{g}$ (simply called \emph{base} in \cite[\S 3]{SerganovaSuperKacMoody}) is a sequence $\Pi'=(\beta_1,\dots,\beta_n)$ of linearly independent roots such that there are $e'_i\in\lie{g}_{\beta_i}$, $f'_i\in\lie{g}_{-\beta_i}$ that together with $\lie{h}$ generate $\lie{g}$ and satisfy $[e'_i,f'_j]=0$ for $i\neq j$.
			The $\beta_i$ are called \emph{simple roots}, and every root can be written as a $\ints$-linear combination of the simple roots such that all the coefficients are either non-negative or non-positive.
			
			A choice of Chevalley generators $e'_i$, $f'_i$ for an ordered root base determines a Cartan datum $(B',\tau')$ by $\tau'_i=|e'_i|$ and $b'_{ij}=\beta_j(h'_i)$ with $h'_i=[e'_i,f'_i]\in\lie{h}$. 
			Observe that the rank of $B$ and $B'$ coincide and thus this gives rise to an isomorphism $\lie{g}\cong\lie{g}(B',\tau')$.
			Note that the $e'_i$ and $f'_i$ are unique up to scalar, so the Cartan datum $(B',\tau')$ is unique up to rescaling of the rows of $B'$.
			
			Let $X$ be a labelling set for the ordered root bases of $\lie{g}$.
			For each $x\in X$ fix Chevalley generators corresponding to the simple roots in $\Pi(x)$ and let $(B(x),\tau(x))$ be the Cartan datum obtained from these.
			\begin{dfn}
				\label{contragredientcartandef}
				The \emph{Cartan graph} $\cartangr_\lie{g}$ of $\lie{g}$ is the $(I,\GCM_I(\ints))$-labelled graph consisting of
				\begin{itemize}
					\item the set of vertices $X$,
 					\item edges according to the rules:
						\begin{itemize}
							\item For each odd isotropic root $\alpha^x_i\in\Pi(x)$ and $\Pi(x')$ obtained from $\Pi(x)$ by an odd reflection at $\alpha^x_i$ (as defined in \cref{oddreflectionformula}), there is an edge $x\overset{i}{\longleftrightarrow}x'$ of color $i\in I$.
							\item For each root $\alpha^x_i\in\Pi(x)$ that is not odd isotropic there is an edge $x\overset{i}{\longleftrightarrow}x$ of color $i\in I$.
						\end{itemize}
					\item the Serre matrices $A(x)=A^{B(x),\tau(x)}$.
				\end{itemize}
				The \emph{Weyl groupoid} of $\lie{g}$ is the Weyl groupoid of $\cartangr_\lie{g}$.
			\end{dfn}
			\begin{rem}
				A priori the Cartan graph depends on the choice of Chevalley generators.
				However, as mentioned above, these are unique up to scalar.
				Rescaling the Chevalley generators corresponds to rescaling the rows of $B(x)$, and this does not affect the Serre matrix $A^{B(x),\tau(x)}$.
				Thus $\cartangr_\lie{g}$ is well-defined.
			\end{rem}
			It is not obvious that $\cartangr_\lie{g}$ is indeed a Cartan graph (or even a semi-Cartan graph) as the name suggests, this will be checked in \cref{itsacartan} below.
			For this we first have to show that the above definition of $\cartangr$ is equivalent to the construction from \cite[Def.~11.2.6]{HeckenbergerSchneider}.
			The idea to associate a Weyl groupoid to a contragredient Lie superalgebra in this way goes back to \cite[Ex.~3]{HeckenbergerYamane}.
			\begin{rem}
				\label{identifylattices}
				For the Cartan graph $\cartangr_\lie{g}$ the basis vectors $\alpha^x_i$ from \cref{weylgroupoiddef} can be identified with the simple roots in the ordered root base $\Pi(x)\subseteq\lie{h}^*$, and in general many of these coincide when viewed as elements of $\lie{h}^*$.
				However, the simple roots in different ordered root bases should always be distinguished.
				In terms of the Cartan graph, this corresponds to distinguishing the simple roots at different vertices as explained in \cref{importanceoflabelatbasis}.
			\end{rem}
			To see that $\cartangr_\lie{g}$ is indeed the same object as the one constructed in \cite{HeckenbergerSchneider} we need to determine the effect of odd reflections on a symmetric Cartan datum.
			Similar formulas (without the symmetry assumption) can also be found in \cite[\S 2.2.4]{GHS} and \cite[\S 4]{HoytSerganova} (for further context see also \cite{AndruskiewitschAngionoSurvey}).
			\begin{prop}
				\label{oddreflectioncartan}
				Let $\Pi$ be an ordered root basis of $\lie{g}$ and $\alpha_i\in\Pi$ odd isotropic.
				Let $\Pi'$ be obtained from $\Pi$ by odd reflection at $\alpha_i$.
				Suppose there are Chevalley generators for $\Pi$ such that in the resulting Cartan datum $(B,\tau)$ the matrix $B$ is symmetric.
				Then there is a choice of Chevalley generators for $\Pi'$ such that corresponding Cartan datum $(B',\tau')$ is given by
				\begin{align*}
					b'_{jk}&=\begin{cases}
						-b_{jk}&\text{if }j=i\text{ or }k=i,\\
						b_{jk}&\text{if }j,k\neq i, b_{ji}b_{ik}=0,\\
						b_{jk}+b_{ik}+b_{ji}&\text{if }j,k\neq i, b_{ji}b_{ik}\neq 0,
					\end{cases}&
					\tau'_j&=\begin{cases}
						\tau_j&\text{if }b_{ij}=0,\\
						\tau_j+\odd&\text{if }b_{ij}\neq 0.
					\end{cases}
				\end{align*}
				In particular any Cartan datum obtained from a symmetrizable Cartan datum under odd reflections is symmetrizable.
			\end{prop}
			\begin{proof}
				Recall from \cref{oddreflectionformula} that the simple roots in $\Pi'$ are
				\begin{equation*}
					\{-\alpha_i\}\cup\{\alpha_j\mid j\neq i, b_{ij}=0\}\cup\{\alpha_j+\alpha_i\mid j\neq i, b_{ij}\neq 0\}.
				\end{equation*}
				One possible choice for the corresponding Chevalley generators is
				\begin{align*}
					e'_j&=\begin{cases}
						f_i&\text{if }j=i,\\
						e_j&\text{if }j\neq i, b_{ij}=0,\\
						[e_i,e_j]&\text{if }j\neq i, b_{ij}\neq 0,
					\end{cases}&
					f'_j&=\begin{cases}
						-e_i&\text{if }j=i,\\
						f_j&\text{if }j\neq i, b_{ij}=0,\\
						\frac{1}{b_{ij}}(-1)^{\tau_j}[f_i,f_j]&\text{if }j\neq i, b_{ij}\neq 0.
					\end{cases}
				\end{align*}
				Observe that this choice of root vectors is unique up to scalar since the root spaces for simple roots and sums of two simple roots are $1$-dimensional, and therefore any other choice of root vectors results in a rescaling of the matrix $B'$.
				Our choice of scaling ensures that $B'$ will be symmetric.

				From the root vectors we compute (using in particular that $e_i$ and $f_i$ are odd, and that $B$ is symmetric)
				\begin{equation*}
					h'_j=[e'_j,f'_j]=\begin{cases}
						-h_i&\text{if }j=i,\\
						h_j&\text{if }j\neq i, b_{ij}=0,\\
						h_i+h_j&\text{if }j\neq i, b_{ij}\neq 0,
					\end{cases}
				\end{equation*}
				and this implies
				\begin{equation*}
					b'_{jk}=\begin{cases}
						-b_{jk}&\text{if }j=i\text{ or }k=i,\\
						b_{jk}&\text{if }j,k\neq i, b_{ji}b_{ik}=0,\\
						b_{jk}+b_{ik}+b_{ji}&\text{if }j,k\neq i, b_{ji}b_{ik}\neq 0
					\end{cases}
				\end{equation*}
				as claimed.
				Finally, $\tau'_j=|e'_j|=|e_j|=\tau_j$ unless $b_{ij}\neq 0$, in which case $\tau'_j=|e'_j|=|e_j|+|e_i|=|e_j|+\odd=\tau_j+\odd$.
			\end{proof}
			\begin{cor}
				\label{itsacartan}
				Let $(B,\tau)$ be a regular symmetrizable Cartan datum, $\lie{g}=\lie{g}(B,\tau)$, and $\tilde{B}$ a symmetrization of $B$.
				Then $\cartangr_\lie{g}$ is precisely the object constructed from $(\tilde{B},\tau)$ in \cite[Def.~11.2.6]{HeckenbergerSchneider}.
				In particular $\cartangr_\lie{g}$ is a Cartan graph.
			\end{cor}
			\begin{proof}
				In \cite{HeckenbergerSchneider}, the Weyl groupoid is defined using the Lie superalgebra $\lie{g}'(B,\tau)\subseteq\lie{g}(B,\tau)$ generated by the $e_i$ and $f_i$.
				The only difference is that the Cartan subalgebra of $\lie{g}'(B,\tau)$ is just spanned by the $h_i$, and therefore this does not affect the construction of the Weyl groupoid.
				
				The definitions then agree by the observation that the Serre matrix is invariant under rescaling of rows of the matrix $B$, and that the effect of odd reflections on a symmetric Cartan datum determined in \cref{oddreflectioncartan} agrees with the formulas from \cite[Lem.~11.2.7]{HeckenbergerSchneider}.
				That $\cartangr_\lie{g}$ is a Cartan graph is then \cite[Thm.~11.2.10]{HeckenbergerSchneider}.
			\end{proof}
		\subsection{Automorphisms}
			We would like to compare the automorphism group of an object of the Weyl groupoid $\weylgrpd$ of a contragredient Lie superalgebra $\lie{g}$ with its Weyl group.
			
			In \cite{GHS} the Weyl group of a connected component of $\weylgrpd$ is introduced.
			It follows from the observations in \cref{groupoidcomparison} that the roots in a connected component of the Cartan graph $\cartangr_\lie{g}$ coincide with the real roots in a connected component of the spine of the root groupoid, as defined in \cite[Def.~4.1.2]{GHS}.
			By \cite[Prop.~4.3.12]{GHS} the Weyl group of a connected component of $\weylgrpd$ is generated by the reflections at non-isotropic roots that appear as simple roots in some ordered root base in this connected component.
			
			In \cite[\S 4]{SerganovaSuperKacMoody} the Weyl group of $\lie{g}$ is defined as the subgroup of $\GL(\lie{h}^*)$ generated by all reflections at all principal even roots of $\lie{g}$, where an even root $\alpha$ is \emph{principal} if either $\alpha$ or $\frac{1}{2}\alpha$ appears as a simple root for $\lie{g}$ in some ordered root base.
			Note that this definition does not depend on the connected component of $\cartangr_\lie{g}$.
			\begin{prop}
				\label{weylgroupoidautos}
				Let $\lie{g}$ be a contragredient Lie superalgebra and $\weylgrpd$ its Weyl groupoid.
				For any object $x\in\weylgrpd$ the group $\Aut_\weylgrpd(x)$ is isomorphic to the Weyl group $W_x$ of the connected component containing $x$.
				
				In particular if for every even root $\alpha$ either $\alpha$ or $\frac{1}{2}\alpha$ appears in an ordered root base in this connected component, then $\Aut_\weylgrpd(\lie{b})$ is isomorphic to the Weyl group of $\lie{g}$.
			\end{prop}
			\begin{proof}
				By definition the Weyl group $W_x$ is a subgroup of $\GL(\lie{h}^*)$.
				On the other hand, by identifying the $\ints$-lattices in the definition of $\weylgrpd$ with the $\ints$-lattice in $\lie{h}^*$ spanned by the roots as in \cref{identifylattices} we can also see $\Aut_\weylgrpd(x)$ as a subgroup of $\GL(\lie{h}^*)$.
				We claim that the respective generators of these groups act by the same reflections on $\lie{h}^*$.
				
				For $x'\in\weylgrpd$ and an odd isotropic simple root $\alpha^{x'}_i\in\Pi(x')$ let $r_i(\Pi(x'))$ the ordered root base obtained from $\Pi(x')$ by odd reflection at $\alpha^{x'}_i$.
				By construction the corresponding generator $(s_i)_{x'}$ of $\weylgrpd$ only does an explicit base change between the bases of $\lie{h}^*$ given by the simple roots $\Pi(x')$ and $r_i(\Pi(x'))$, and hence acts as the identity map on $\lie{h}^*$.
				But this means that $\Aut_\weylgrpd(x)$ is generated by the reflections at all non-isotropic roots that appear as a simple root in some ordered root base.
				As the formulas defining the reflections are the same in both cases (see \cref{weylgroupoiddef} and \cite[\S 4]{SerganovaSuperKacMoody}, \cite[Eq.~(7)]{GHS}) we see that $\Aut_\weylgrpd(x)\subseteq W_x$.
				
				The converse inclusion is clear by definition of $W_x$.
			\end{proof}
			\begin{rem}
				From \cref{weylgroupoidautos} it follows that the real roots of $\weylgrpd$ in the sense of \cref{definitionrealroots} are the same as the real roots of $\lie{g}$.
			\end{rem}
		\subsection{Components of the Cartan graph}
			\label{cartancomponents}
			In general, the Cartan graph of a contragredient Lie superalgebra will have many connected components, see \cref{borelevenpart} below.
			However in some cases it is enough to consider only one of these.
			\begin{rem}
				\label{borelevenpart}
				An ordered root base $\Pi'$ of $\lie{g}$ determines a decomposition $\lie{g}=\lie{n}'^-\oplus\lie{h}\oplus\lie{n}'^+$ into a positive and negative part with respect to its simple roots.
				By slight abuse of language we call $\lie{b}'=\lie{h}\oplus\lie{n}'^+$ a \emph{Borel subalgebra} of $\lie{g}$.
				
				Since odd reflections do not change the even part of a Borel subalgebra, the Cartan graph of a contragredient Lie superalgebra splits into several components without edges between them.
				If the Borel subalgebras with the same even part represent all conjugacy classes of Borel subalgebras, then these components all look the same and we restrict our attention to one of these components.
				This is for instance the case for $\lie{sl}(m|n)$, $\lie{osp}(2m+1|2n)$ and $\lie{osp}(2m|2n)$, see e.g.~\cite[\S 3.1]{Musson}.
			\end{rem}
			\begin{rem}
				\label{rootsordering}
				From \cite[Thm.~9.3.5]{HeckenbergerSchneider} it follows that it is possible to order the simple roots consistently under odd reflections in the sense that for an ordered root base $\Pi=(\alpha_1,\dots,\alpha_n)$ and a non-trivial reordering $\Pi'$ of $\Pi$ it is impossible to obtain $\Pi'$ from $\Pi$ by odd reflections.
				Therefore the Cartan graph $\cartangr_\lie{g}$ decomposes into $n!$ identical (up to renumbering of edges) components without edges between them.
				However, this uses that $\cartangr_\lie{g}$ is a Cartan graph, and therefore we cannot choose a consistent ordering of the simple roots a priori.
				As far as we know, a consistent ordering of the simple roots cannot be found purely in terms of root combinatorics of Lie superalgebras,  although its existence is a purely Lie-theoretical question.
			\end{rem}
		\subsection{Relation to other notions of Weyl groupoids}
			\label{groupoidcomparison}
			There are several constructions called \emph{Weyl groupoid} in the literature.
			Our definition of the Weyl groupoid $\weylgrpd$ generalizes the construction from \cite{HeckenbergerYamane}.
			The relation to the other notions is as follows.
			
			In \cite{SerganovaSuperKacMoody} Serganova introduced another notion of \emph{Weyl groupoid $\mathcal{C}$} whose objects are Cartan data $(B,\tau)$ and whose morphisms are isomorphisms $\lie{g}(B,\tau)\to\lie{g}(B',\tau')$ of the associated contragredient Lie superalgebras that preserve the Cartan subalgebra.
			In virtue of \cite[Thm.\ 4.14]{SerganovaSuperKacMoody} this compares to $\weylgrpd$ as follows.
			\begin{comp}
				The Weyl groupoid $\weylgrpd$ is the subgroupoid of the component of $\mathcal{C}$ containing $\lie{g}(B,\tau)$, obtained by forgetting all morphisms corresponding to rescaling rows of the matrices $B$.
			\end{comp}
			One could say that this is a restriction to the essentially important information as rescaling rows only amounts to different choices of Chevalley generators.

			In a recent preprint \cite{GHS}, Gorelik, Hinich and Serganova constructed a different version of a Weyl groupoid called \emph{root groupoid}.
			For a fixed finite set $X$, their objects are quadruples $(\lie{h},a,b,p)$ where $\lie{h}$ denotes a Cartan subalgebra, $a\colon X\to\lie{h}$ a map with image a set of linearly independent coroots, $b\colon X\to\lie{h}^*$ a map with image a set of linearly independent roots and $p$ the corresponding parities.
			The root groupoid is generated by the following three types of morphisms:
			\begin{itemize}
				\item $(\lie{h},a,b,p)\to(\lie{h}',\theta\circ a,\theta^{-1}\circ b,p)$ for any isomorphism $\lie{h}\overset{\cong}{\rightarrow}\lie{h}'$,
				\item $(\lie{h},a,b,p)\to(\lie{h},a',b,p)$, where $a'(x)=\lambda(x)a(x)$ for some $\lambda\colon X\to\compl^*$,
				\item even and odd reflections.
			\end{itemize}
			In \cite[\S 4.2.5]{GHS} the \emph{skeleton} of the root groupoid is defined as the subgroupoid generated by the even and odd reflections.
			Furthermore, the \emph{spine} of the root groupoid is defined as the subgroupoid generated by odd reflections only, see \cite[\S 4.2.8]{GHS}.

			Given a regular symmetrizable Cartan datum $(B,\tau)$ of rank $n$, we can choose a minimal realization of $\lie{h}$, i.e.\  a vector space $\lie{h}$ together with linearly independent coroots $a_1, \dots a_n$ and linearly independent roots $b_1, \dots, b_n$ such that the natural pairing satisfies $\langle a_i, b_j\rangle = B_{ij}$.
			Thus we obtain a quadruple $v=(\lie{h}, a,b,\tau)$.
			The connected component of this quadruple in the root groupoid is an admissible, fully reflectable component in the sense of \cite[Def.\ 3.2.3, \S 3.4.1]{GHS}.
			\begin{comp}
				The connected component of $v$ in the skeleton of the root groupoid is the simply connected cover of $\weylgrpd$ in the sense of \cite[Def.\ 9.1.10 and 10.1.1]{HeckenbergerSchneider}.
			\end{comp}
			\begin{comp}
				The subgroupoid $\weylgrpd'$ of $\weylgrpd$ generated by all isotropic reflections is isomorphic to the connected component of $v$ in the spine of the root groupoid.
			\end{comp}
			Yet another notion of Weyl groupoids was suggested by Sergeev and Veselov in \cite{SergeevVeselovGrothendieckRing}.
			However as they remark this construction is completely unrelated to the notion of Weyl groupoid we work with.
	\section{Borel subalgebras of \texorpdfstring{$\lie{sl}(m|n)$}{sl(m|n)}, \texorpdfstring{$\lie{osp}(2m+1|2n)$}{osp(2m|2n)} and \texorpdfstring{$\lie{osp}(2m|2n)$}{osp(2m|2n)}}
		\label{borelsandcartan}
		To give a detailed description of the Weyl groupoids of $\lie{sl}(m|n)$, $\lie{osp}(2m+1|2n)$ and $\lie{osp}(2m|2n)$ we first need some preparation.
		Recall that the Lie superalgebra $\lie{sl}(m|n)$ is given in matrices by 
		\[
			\left(\begin{array}{c|c}
			A&B\\\hline C&D
			\end{array}\right)
		\]
		such that $\tr(A)-\tr(D)=0$ together with the usual supercommutator $[x,y]=x\circ y - (-1)^{\abs{x}\abs{y}}y\circ x$.
		
		The orthosymplectic Lie superalgebra $\lie{osp}(2m+1|2n)$ is explicitly given by all matrices of the form
		\begin{equation}\label{dfnospexplicit}	
			\left(\begin{array}{ccc|cc}
				0&-u^t&-v^t&x&x_1\\
				v&a&b&y&y_1\\
				u&c&-a^t&z&z_1\\
				\hline
				-x_1^t&-z_1^t&-y_1^t&d&e\\
				x^t&z^t&y^t&f&-d^t
			\end{array}\right)
		\end{equation}
		where $a$ is any $(m\times m)$-matrix; $b$ and $c$ are skew-symmetric $(m\times m)$-matrices; $d$ is any $(n\times n)$-matrix; $e$ and $f$ are symmetric $(n\times n)$-matrices; $u$ and $v$ are $(m\times 1)$-matrices; $y$, $y_1$, $z$ and $z_1$ are $(m\times n)$-matrices; and $x$ as well as $x_1$ are $(1\times n)$-matrices.
		The Lie superalgebra $\lie{osp}(2m|2n)$ is given by the same matrices, except that we have to delete the first row and column.
		We label the rows and columns by $0,1,\dots,m,-1,\dots,-m,(m+1),\dots,(m+n),-(m+1),\dots,-(m+n)$ in this order (leaving out $0$ for $\lie{osp}(2m|2n)$).
		
		To determine the Weyl groupoids of $\lie{sl}(m|n)$, $\lie{osp}(2m+1|2n)$ and $\lie{osp}(2m|2n)$ we require an explicit description of all the possible realizations as contragredient Lie superalgebras.
		Hence we need to determine all their ordered root bases, which mostly amounts to determining their Borel subalgebras, and the corresponding Cartan data.
		These results are standard and are essentially already contained in \cite[\S 2.5.5]{Kac}, though the formulation there is less convenient and not explicit enough for our purposes.
		
		For $\lie{sl}(n|n)$ there is a minor extra difficulty as the simple roots are not linearly independent, and the construction of the contragredient Lie superalgebra from the Cartan data computed from $\lie{sl}(n|n)$ yields $\lie{gl}(n|n)$ rather than $\lie{sl}(n|n)$.
		Nevertheless we will use $\lie{sl}(n|n)$ in all computations below, as all statements about Borel subalgebras and simple roots carry over to $\lie{gl}(n|n)$.
		
		Let $\lie{g}$ be either $\lie{sl}(m|n)$, $\lie{osp}(2m+1|2n)$ or $\lie{osp}(2m|2n)$.
		Recall (see e.g.~\cite[\S 3.1]{Musson}) that for a fixed Borel subalgebra $\lie{b}_\even\subseteq\lie{g}_\even$ there are only finitely many Borel subalgebras $\lie{b}\subseteq\lie{g}$ with even part $\lie{b}_\even$.
		Moreover these Borel subalgebras can be relatively easily described in terms of partitions fitting in an $m\times n$-rectangle, see e.g.~\cite[Proposition 3.3.8]{Musson}:
		For $\lie{sl}(m|n)$ and $\lie{osp}(2m+1|2n)$ the Borel subalgebras with fixed even part are in bijection with the set of partitions $\lambda$ fitting in an $m\times n$-rectangle.
		For $\lie{osp}(2m|2n)$ each partition $\lambda$ fitting in an $m\times n$-rectangle determines two Borel subalgebras $\lie{b}(\lambda,+)$ and $\lie{b}(\lambda,-)$, which coincide if and only if $\lambda_1=n$.
		
		In \cref{constructionBorelsl,constructionBoreloddosp,constructionBorelevenosp} below we provide a detailed description of the Borel subalgebras of $\lie{sl}(m|n)$, $\lie{osp}(2m+1|2n)$ and $\lie{osp}(2m|2n)$, based on the description in \cite[\S 3.3--3.4]{Musson}.
		We compute the corresponding Cartan data in \cref{cartancomputation}.
		For this it is more convenient to work with permutations instead of partitions, and therefore we will frequently use \cref{shufflespartitions} to pass between these.
		\subsection{Borel subalgebras for \texorpdfstring{$\lie{sl}(m|n)$}{sl(m|n)}}\label{constructionBorelsl}
			For the Lie superalgebra $\lie{sl}(m|n)$ we fix the $m+n-1$-dimensional Cartan subalgebra $\lie{h}$ given by all the diagonal matrices, and we are interested in Borel subalgebras with even part $\lie{b}_\even$ given by the standard even Borel subalgebra of upper triangular matrices.
			As usual, we let $\eps_i\in\lie{h}^*$ ($1\leq i\leq m+n$) denote the projection to the $i$-th diagonal entry.
			
			Given a partition $\lambda\in\partitions_{m\times n}$ we can construct the odd part of a Borel subalgebra $\lie{b}(\lambda)$ with even part $\lie{b}_\even$ as follows:
			Draw $\lambda$ in an $m\times n$-rectangle as in \cref{partitioninrectangle}, which we identify with the top right $m\times n$-block.
			The entries corresponding to the boxes of $\lambda$ are required to be $0$, while the other entries in the top right block can be arbitrary.
			Similarly drawing the transpose of the complement of $\lambda$ (taken in the $m\times n$-rectangle) into the lower right $n\times m$-block determines the zeros and arbitrary entries there.
			For instance the ``standard'' Borel subalgebra of upper triangular matrices corresponds to $\lambda=\emptyset$, and for another concrete example see \cref{partitionborelexamplesl} below.
			By \cite[Prop.~3.3.8]{Musson} the $\lie{b}(\lambda)$'s are all the Borel subalgebras of $\lie{sl}(m|n)$ with even part $\lie{b}_\even$.
			\begin{ex}\label{partitionborelexamplesl}
				Let  $m=3$, $n=4$ and $\lambda=(4,2,1)$.
				Then the corresponding Borel subalgebra of $\lie{sl}(3|4)$ is given by
				\begin{equation}
					\lie{b}(\lambda)=\left(
						\begin{tikzpicture}[baseline=(current bounding box.center),scale=0.5,text height=1ex, text width=1ex, text depth=0ex]
							\matrix (slmatr) [matrix of math nodes,column sep=0.1cm, row sep=0.1cm]
							{
								*&*&* &0&*&*&*\\
								0&*&* &0&0&*&*\\
								0&0&* &0&0&0&0\\
								*&*&* &*&*&*&*\\
								0&*&* &0&*&*&*\\
								0&0&* &0&0&*&*\\
								0&0&* &0&0&0&*\\
							};
							\draw ($(slmatr-1-3.north east)!0.5!(slmatr-1-4.north west)$) -- ($(slmatr-7-3.south east)!0.5!(slmatr-7-4.south west)$);
							\draw ($(slmatr-3-1.south west)!0.5!(slmatr-4-1.north west)$) -- ($(slmatr-3-7.south east)!0.5!(slmatr-4-7.north east)$);
							
							\draw (slmatr-1-4.north west) -- ($(slmatr-1-4.north east)!0.5!(slmatr-1-5.north west)$) -- ($(slmatr-1-4.south east)!0.5!(slmatr-2-5.north west)$) -- ($(slmatr-2-5.north east)!0.5!(slmatr-1-6.south west)$) -- ($(slmatr-2-6.south west)!0.5!(slmatr-3-5.north east)$) -- ($(slmatr-2-7.south east)!0.5!(slmatr-3-7.north east)$) -- (slmatr-3-7.south east);
							\draw (slmatr-4-1.north west) -- ($(slmatr-4-1.south west)!0.5!(slmatr-5-1.north west)$) -- ($(slmatr-4-2.south west)!0.5!(slmatr-5-1.north east)$) -- ($(slmatr-5-2.south west)!0.5!(slmatr-6-1.north east)$) -- ($(slmatr-5-3.south west)!0.5!(slmatr-6-2.north east)$) -- ($(slmatr-7-2.south east)!0.5!(slmatr-7-3.south west)$) -- (slmatr-7-3.south east);
						\end{tikzpicture}
					\right)
				\end{equation}
			\end{ex}
			Using \cref{shufflespartitions} to pass between shuffles and partitions, we can explicitly describe the ordered root bases and Cartan data for the Borel subalgebras of $\lie{sl}(m|n)$, see \cref{slmnchevalley}
		\subsection{Borel subalgebras for \texorpdfstring{$\lie{osp}(2m+1|2n)$}{osp(2m+1|2n)}}\label{constructionBoreloddosp}
			To describe all Borel subalgebras of $\lie{osp}(2m+1|2n)$ we fix the Cartan subalgebra $\lie{h}$ consisting of the diagonal matrices and let $\eps_i\in\lie{h}^*$ ($i\in \{\pm 1,\dots,\pm (m+n)\}$) denote the projections to the diagonal entries.
			We also fix the standard Borel subalgebra $\lie{b}_\even$ of the even part, which is defined by the simple roots $\eps_i-\eps_{i+1}$ for $1\leq i\leq m-1$ as well as $\eps_m$ and $\eps_{m+j}-\eps_{m+j+1}$ for $1\leq j\leq n-1$ together with $2\eps_{m+n}$.
			
			For a partition $\lambda\in\partitions_{m\times n}$ we can construct the odd part of a Borel subalgebra $\lie{b}(\lambda)$ with even part $\lie{b}_\even$ as follows:
			In the notation from \cref{dfnospexplicit} we demand that $z=0$ and $x=0$, while $x_1$ and $y_1$ can be chosen arbitrarily.
			Note that $y$ and $z_1$ are $m\times n$-matrices, and we identify these with the $m\times n$-rectangle from \cref{shufflespartitions}.
			For $y$, the entries in the boxes corresponding to $\lambda$ must be zero while the other entries are arbitrary, and for $z_1$ the rule is exactly the opposite.
			Again by \cite[Prop.~3.3.8]{Musson} the $\lie{b}(\lambda)$'s are all the Borel subalgebras of $\lie{osp}(2m+1|2n)$ with even part $\lie{b}_\even$.
			\begin{ex}
				Let $m=1$, $n=2$ and consider the partition $\lambda=(1)$.
				The corresponding Borel subalgebra of $\lie{osp}(3|4)$ is given by
				\begin{equation*}
					\lie{b}(\lambda)=\left(
					\begin{tikzpicture}[baseline=(current bounding box.center),scale=0.5,text height=1ex, text width=1ex, text depth=0ex]
						\matrix (ospmatr) [matrix of math nodes,column sep=0.1cm, row sep=0.1cm]
						{
							0&0&* &0&0&*&*\\
							*&*&0 &0&*&*&*\\
							0&0&* &0&0&*&0\\
							*&*&* &*&*&*&*\\
							*&0&* &0&*&*&*\\
							0&0&0 &0&0&*&0\\
							0&0&* &0&0&*&*\\
						};
						\draw ($(ospmatr-1-3.north east)!0.5!(ospmatr-1-4.north west)$) -- ($(ospmatr-7-3.south east)!0.5!(ospmatr-7-4.south west)$);
						\draw ($(ospmatr-3-1.south west)!0.5!(ospmatr-4-1.north west)$) -- ($(ospmatr-3-7.south east)!0.5!(ospmatr-4-7.north east)$);
						
						\draw[dashed] (ospmatr-2-4.north west) rectangle (ospmatr-2-5.south east);
						\draw[dashed] (ospmatr-3-6.north west) rectangle (ospmatr-3-7.south east);
						\draw[dashed] (ospmatr-4-2.north west) rectangle (ospmatr-5-2.south east);
						\draw[dashed] (ospmatr-6-3.north west) rectangle (ospmatr-7-3.south east);
						
						\draw (ospmatr-2-4.north west) -- ($(ospmatr-2-4.north east)!0.5!(ospmatr-2-5.north west)$) -- ($(ospmatr-2-4.south east)!0.5!(ospmatr-2-5.south west)$) -- (ospmatr-2-5.south east);
						\draw (ospmatr-3-6.north west) -- ($(ospmatr-3-6.north east)!0.5!(ospmatr-3-7.north west)$) -- ($(ospmatr-3-6.south east)!0.5!(ospmatr-3-7.south west)$) -- (ospmatr-3-7.south east);
						\draw (ospmatr-4-2.north west) -- ($(ospmatr-4-2.south west)!0.5!(ospmatr-5-2.north west)$) -- ($(ospmatr-4-2.south east)!0.5!(ospmatr-5-2.north east)$) -- (ospmatr-5-2.south east);
						\draw (ospmatr-6-3.north west) -- ($(ospmatr-6-3.south west)!0.5!(ospmatr-7-3.north west)$) -- ($(ospmatr-6-3.south east)!0.5!(ospmatr-7-3.north east)$) -- (ospmatr-7-3.south east);
					\end{tikzpicture}
					\right)
				\end{equation*}
			\end{ex}
			The corresponding ordered root bases and Cartan data are described in terms of shuffles in \cref{oddospchevalley}.
		\subsection{Borel subalgebras for \texorpdfstring{$\lie{osp}(2m|2n)$}{osp(2m|2n)}}\label{constructionBorelevenosp}
			The case of $\lie{osp}(2m|2n)$ is slightly more involved.
			Again we fix the Cartan subalgebra $\lie{h}$ consisting of diagonal matrices and let $\eps_i\in\lie{h}^*$ denote the projection to the $i$-th diagonal entry (with the same ordered basis and index conventions as for $\lie{osp}(2m+1|2n)$).
			Furthermore, we fix the standard Borel $\lie{b}_\even$ of the even part, which is given by the simple roots $\eps_i-\eps_{i+1}$ for $1\leq i\leq m-1$ as well as $\eps_{m-1}+\eps_m$ and $\eps_{m+j}-\eps_{m+j+1}$ for $1\leq j\leq n-1$ together with $2\eps_{m+n}$.
			
			Suppose we are given a partition $\lambda\in\partitions_{m\times n}$ and $\eps\in\{+,-\}$.
			From this we construct the odd part of a Borel subalgebra $\lie{b}(\lambda,\eps)$ with even part $\lie{b}_\even$ as follows.
			If $\eps=+$, the entries are determined by $\lambda$ by the same rules as in \cref{constructionBoreloddosp}.
			If $\eps=-$, we do the same construction as for $+$ but afterwards we swap the $m$-th and the $(-m)$-th row of the top right block.
			The $\lie{b}(\lambda,\eps)$'s are all the Borel subalgebras of $\lie{osp}(2m|2n)$ with even part $\lie{b}_\even$ by \cite[Prop.~3.3.8]{Musson}.
			
			Observe that $\lie{b}(\lambda,+)=\lie{b}(\lambda,-)$ if and only if $\lambda_1=n$, since we need the $m$-th row of $y$ to be zero and the $m$-th row of $z_1$ to be arbitrary.
			In this case we also denote the resulting Borel subalgebra by $\lie{b}(\lambda,\pm)$.
			\begin{ex}
				Let $m=n=2$.
				From the partitions $\lambda=(1,1)$ and $\mu=(2,1)$ we obtain the following Borel subalgebras $\lie{b}(\lambda,+)$, $\lie{b}(\lambda,-)$ and $\lie{b}(\mu,\pm)$ of $\lie{osp}(4|4)$.
				Here $\lie{b}(\lambda, -)$ is obtained from $\lie{b}(\lambda, +)$ by swapping the rows and columns as indicated.
				\begin{figure}[H]
				\begin{subfigure}[t]{0.3\linewidth}
						$\left(
						\begin{tikzpicture}[ampersand replacement=\&,baseline=(current bounding box.center),text height=1ex, text width=1ex, text depth=0ex, scale=0.75, every node/.style={scale=0.75}]
							\matrix (ospmatr) [matrix of math nodes,column sep=0.1cm,row sep=0.1cm]
							{
								*\&*\&0\&* \&0\&*\&*\&*\\
								0\&*\&*\&0 \&0\&*\&*\&*\\
								0\&0\&*\&0 \&0\&0\&*\&0\\
								0\&0\&*\&* \&0\&0\&*\&0\\
								*\&*\&*\&* \&*\&*\&*\&*\\
								0\&0\&*\&* \&0\&*\&*\&*\\
								0\&0\&0\&0 \&0\&0\&*\&0\\
								0\&0\&*\&* \&0\&0\&*\&*\\
							};
							\draw ($(ospmatr-1-4.north east)!0.5!(ospmatr-1-5.north west)$) -- ($(ospmatr-8-4.south east)!0.5!(ospmatr-8-5.south west)$);
							\draw ($(ospmatr-4-1.south west)!0.5!(ospmatr-5-1.north west)$) -- ($(ospmatr-4-8.south east)!0.5!(ospmatr-5-8.north east)$);
							
							\draw[dashed] (ospmatr-1-5.north west) rectangle (ospmatr-2-6.south east);
							\draw[dashed] (ospmatr-3-7.north west) rectangle (ospmatr-4-8.south east);
							\draw[dashed] (ospmatr-5-1.north west) rectangle (ospmatr-6-2.south east);
							\draw[dashed] (ospmatr-7-3.north west) rectangle (ospmatr-8-4.south east);
							
							\draw (ospmatr-1-5.north west) -- ($(ospmatr-1-5.north east)!0.5!(ospmatr-1-6.north west)$) -- ($(ospmatr-2-5.south east)!0.5!(ospmatr-2-6.south west)$) -- (ospmatr-2-6.south east);
							\draw (ospmatr-3-7.north west) -- ($(ospmatr-3-7.north east)!0.5!(ospmatr-3-8.north west)$) -- ($(ospmatr-4-7.south east)!0.5!(ospmatr-4-8.south west)$) -- (ospmatr-4-8.south east);
							\draw (ospmatr-5-1.north west) -- ($(ospmatr-5-1.south west)!0.5!(ospmatr-6-1.north west)$) -- ($(ospmatr-5-2.south east)!0.5!(ospmatr-6-2.north east)$) -- (ospmatr-6-2.south east);
							\draw (ospmatr-7-3.north west) -- ($(ospmatr-7-3.south west)!0.5!(ospmatr-8-3.north west)$) -- ($(ospmatr-7-4.south east)!0.5!(ospmatr-8-4.north east)$) -- (ospmatr-8-4.south east);
							
							\pgfresetboundingbox 
							\useasboundingbox (ospmatr-1-1.north west) rectangle (ospmatr-8-8.south east);
						\end{tikzpicture}
						\right)$
					\caption*{$\lie{b}(\lambda, +)$}
				\end{subfigure}\hfill
				\begin{subfigure}[t]{0.3\linewidth}
						$\left(
							\begin{tikzpicture}[ampersand replacement=\&,baseline=(current bounding box.center),text height=1ex, text width=1ex, text depth=0ex, scale=0.75, every node/.style={scale=0.75}]
								\matrix (ospmatr) [matrix of math nodes,column sep=0.1cm,row sep=0.1cm]
								{
									*\&*\&0\&* \&0\&*\&*\&*\\
									0\&*\&*\&0 \&0\&0\&*\&0\\
									0\&0\&*\&0 \&0\&0\&*\&0\\
									0\&0\&*\&* \&0\&*\&*\&*\\
									*\&*\&*\&* \&*\&*\&*\&*\\
									0\&*\&*\&0 \&0\&*\&*\&*\\
									0\&0\&0\&0 \&0\&0\&*\&0\\
									0\&*\&*\&0 \&0\&0\&*\&*\\
									};
									\pgfresetboundingbox 
									\useasboundingbox (ospmatr-1-1.north west) rectangle (ospmatr-8-8.south east);
									\draw[thick, densely dotted] (ospmatr-2-1.north west) rectangle (ospmatr-2-8.south east);
									\draw[thick, densely dotted] (ospmatr-4-1.north west) rectangle (ospmatr-4-8.south east);
									\draw[thick, densely dotted] (ospmatr-1-2.north west) rectangle (ospmatr-8-2.south east);
									\draw[thick, densely dotted] (ospmatr-1-4.north west) rectangle (ospmatr-8-4.south east);
									\draw[<->, out = 90, in = 90, shorten >=1pt, shorten <=1pt] (ospmatr-1-2.north) to (ospmatr-1-4.north);
									\draw[<->, out = 0, in = 0, shorten >=1pt, shorten <=1pt] ($(ospmatr-2-8.east)+(8pt, 0)$) to ($(ospmatr-4-8.east)+(8pt,0)$);
									\draw ($(ospmatr-1-4.north east)!0.5!(ospmatr-1-5.north west)$) -- ($(ospmatr-8-4.south east)!0.5!(ospmatr-8-5.south west)$);
									\draw ($(ospmatr-4-1.south west)!0.5!(ospmatr-5-1.north west)$) -- ($(ospmatr-4-8.south east)!0.5!(ospmatr-5-8.north east)$);
								\end{tikzpicture}
								\right)$
								\caption*{$\lie{b}(\lambda, -)$}
					\end{subfigure}
					\hfill
					\begin{subfigure}[t]{0.3\linewidth}
					$\left(
						\begin{tikzpicture}[ampersand replacement=\&,baseline=(current bounding box.center),text height=1ex, text width=1ex, text depth=0ex, scale=0.75, every node/.style={scale=0.75}]
							\matrix (ospmatr) [matrix of math nodes,column sep=0.1cm,row sep=0.1cm]
							{
								*\&*\&0\&* \&0\&*\&*\&*\\
								0\&*\&*\&0 \&0\&0\&*\&*\\
								0\&0\&*\&0 \&0\&0\&*\&0\\
								0\&0\&*\&* \&0\&0\&*\&*\\
								*\&*\&*\&* \&*\&*\&*\&*\\
								0\&*\&*\&* \&0\&*\&*\&*\\
								0\&0\&0\&0 \&0\&0\&*\&0\\
								0\&0\&*\&0 \&0\&0\&*\&*\\
								};
								\pgfresetboundingbox 
								\useasboundingbox (ospmatr-1-1.north west) rectangle (ospmatr-8-8.south east);
								\draw ($(ospmatr-1-4.north east)!0.5!(ospmatr-1-5.north west)$) -- ($(ospmatr-8-4.south east)!0.5!(ospmatr-8-5.south west)$);
								\draw ($(ospmatr-4-1.south west)!0.5!(ospmatr-5-1.north west)$) -- ($(ospmatr-4-8.south east)!0.5!(ospmatr-5-8.north east)$);
								
								\draw[dashed] (ospmatr-1-5.north west) rectangle (ospmatr-2-6.south east);
								\draw[dashed] (ospmatr-3-7.north west) rectangle (ospmatr-4-8.south east);
								\draw[dashed] (ospmatr-5-1.north west) rectangle (ospmatr-6-2.south east);
								\draw[dashed] (ospmatr-7-3.north west) rectangle (ospmatr-8-4.south east);
								
								\draw (ospmatr-1-5.north west) -- ($(ospmatr-1-5.north east)!0.5!(ospmatr-1-6.north west)$) |- ($(ospmatr-1-6.south east)!0.5!(ospmatr-2-6.north east)$) -- (ospmatr-2-6.south east);
								\draw (ospmatr-3-7.north west) -- ($(ospmatr-3-7.north east)!0.5!(ospmatr-3-8.north west)$) |- ($(ospmatr-3-8.south east)!0.5!(ospmatr-4-8.north east)$) -- (ospmatr-4-8.south east);
								\draw (ospmatr-5-1.north west) -- ($(ospmatr-5-1.south west)!0.5!(ospmatr-6-1.north west)$) -| ($(ospmatr-6-1.south east)!0.5!(ospmatr-6-2.south west)$) -- (ospmatr-6-2.south east);
								\draw (ospmatr-7-3.north west) -- ($(ospmatr-7-3.south west)!0.5!(ospmatr-8-3.north west)$) -| ($(ospmatr-8-3.south east)!0.5!(ospmatr-8-4.south west)$) -- (ospmatr-8-4.south east);
							\end{tikzpicture}
							\right)$
							\caption*{$\lie{b}(\mu, \pm)$}
					\end{subfigure}
				\end{figure}
			\end{ex}
			We describe the corresponding ordered root bases and Cartan data in \cref{evenospchevalley}, again in terms of shuffles rather than partitions.
			To connect this to the above description of the Borel subalgebras, observe that if a shuffle $\sigma$ corresponds to a partition $\lambda$ under the bijection from \cref{shufflespartitions}, then $\sigma(m+n)=m$ if and only if $\lambda_1=n$.
			\begin{rem}
				As explained in \cite[\S 3.3]{Musson} the extra difficulties for $\lie{osp}(2m|2n)$ are due to the existence of an outer automorphism of $\lie{o}(2m)$ that on $\lie{h}^*$ swaps $\eps_m$ and $\eps_{-m}=-\eps_m$.
				This corresponds precisely to the swapping of rows in the construction of the Borel subalgebra $\lie{b}(\lambda,-)$ from $\lie{b}(\lambda,+)$.
			\end{rem}
		\subsection{Odd reflections in terms of partitions}
			\label{oddreflectionpartition}
			For $\lie{sl}(m|n)$, $\lie{osp}(2m+1|2n)$ and $\lie{osp}(2m|2n)$ we can describe odd reflections in terms of partitions as follows.
			Consider the $m\times n$-rectangle and number the boxes ascendingly in each row and column, starting from a $1$ in the top left as in the following example.
			\begin{equation*}
				\begin{tikzpicture}
					\fill[lightgray] (0,0) -- (0.5,0) -- (0.5,-0.5) -- (1,-0.5) -- (1,-1) -- (2,-1) -- (2,-1.5) -- (0,-1.5) -- (0,0);
						
					\draw[step=0.5,dashed] (0,0) grid (2,-1.5);
					\draw (0,0) rectangle (2,-1.5);
					
					\draw[thick] (0,0) -- (0.5,0) -- (0.5,-0.5) -- (1,-0.5) -- (1,-1) -- (2,-1) -- (2,-1.5);
					
					\foreach\x in {0,...,3} {
						\foreach\y in {0,...,2} {
							\pgfmathsetmacro\boxnumber{\x+\y+1}
							\node at ($(\x*0.5+0.25,-0.25-\y*0.5)$) {$\pgfmathprintnumber{\boxnumber}$};
						}
					}
				\end{tikzpicture}
			\end{equation*}
			Let $\lambda\in\partitions_{m\times n}$ and let $\lie{b}(\lambda)$ be the Borel subalgebra of $\lie{sl}(m|n)$ or $\lie{osp}(2m+1|2n)$ constructed from $\lambda$.
			Observe that the numbers in the boxes that can be removed from or added to $\lambda$ so that the result is still a partition $\lambda'\in\partitions_{m\times n}$ are precisely the indices of the odd isotropic simple roots for $\lie{b}(\lambda)$.
			In particular such a box (if it exists) is unique.
			The Borel subalgebra obtained from $\lie{b}(\lambda)$ by an odd reflection at the simple root $\alpha_i$ is $\lie{b}(\lambda')$, with $\lambda'\in\partitions_{m\times n}$ obtained from $\lambda$ by adding or removing a box numbered with $i$.
			
			Unsurprisingly the description of odd reflections for $\lie{osp}(2m|2m)$ is slightly more complicated due to the different series of Borels.
			In this case the odd reflection at $\alpha_i$ takes $\lie{b}(\lambda,\eps)$ to $\lie{b}(\lambda',\eps')$ according to the following rules (using the implicit convention $\eps'=\pm$ if $\lambda'_1=n$):
			\begin{itemize}
				\item If $\lambda_1<n$ and $\eps=+$, then $\eps'=+$ and $\lambda'$ is obtained from $\lambda$ by adding or removing a box numbered $i$ (note that in this case $\alpha_{m+n}=2\eps_{m+n}$ is even).
				\item If $\lambda_1<n$ and $\eps=-$, then $\eps'=-$ and $\lambda'$ is obtained from $\lambda$ by adding or removing a box numbered $i$ for $i<m+n-1$, and by adding the box numbered $m+n-1$ for $i=m+n$ (note that $\alpha_{m+n-1}=2\eps_{m+n}$ is even).
				\item If $\lambda_1=n$, then $\eps=\pm$.
					\begin{itemize}
						\item If $i<m+n-1$, then $\eps'=\pm$ and $\lambda'$ is obtained from $\lambda$ by adding or removing a box numbered $i$.
						\item If $i=m+n-1$, then $\eps'=+$ and $\lambda'$ is obtained from $\lambda$ by removing the box numbered $m+n-1$.
						\item If $i=m+n$, then $\eps'=-$ and $\lambda'$ is obtained from $\lambda$ by removing the box numbered $m+n-1$.
					\end{itemize}
			\end{itemize}
	\section{The Weyl groupoids of \texorpdfstring{$\lie{sl}(m|n)$}{sl(m|n)}, \texorpdfstring{$\lie{osp}(2m+1|2n)$}{osp(2m+1|2n)} and \texorpdfstring{$\lie{osp}(2m|2n)$}{osp(2m|2n)}}
		In this section we give a detailed description of the Cartan graphs and the Weyl groupoids of $\lie{sl}(m|n)$, $\lie{osp}(2m+1|2n)$ and $\lie{osp}(2m|2n)$.
		We begin by describing the underlying graph of the Cartan graph and then list the Serre matrices.
		Finally we determine the Coxeter-type relations among the generators of the Weyl groupoids.
		
		These Weyl groupoids also appear in \cite{AndruskiewitschAngionoSurvey}, where they are studied from the perspective of Nichols algebras.
		However, our combinatorics is based on the the graphical description of Borel subalgebras in terms of partitions.
		This directly reflects the structural theory of the Lie superalgebra, and therefore makes it very easy to pass between Weyl groupoids and Lie superalgebras.
		A further advantage of our description is that it is very easy to write down the Weyl groupoids in concrete examples.
		\subsection{Shape of the Cartan graph}
			In \cref{constructionBorelsl,constructionBoreloddosp,constructionBorelevenosp} we gave a detailed description of the (finitely many) Borel subalgebras of $\lie{sl}(m|n)$, $\lie{osp}(2m+1|2n)$ and $\lie{osp}(2m|2n)$ with fixed even part.
			As mentioned above, these represent all conjugacy classes of Borel subalgebras.
			Therefore by \cref{borelevenpart} their Cartan graphs split into several identical subgraphs without edges between each other.
			Hence we only need to consider one of these subgraphs, namely the one corresponding to the Borel subalgebras described above.
			The number of these subgraphs is the order of the Weyl group, i.e.~$m!n!$ for $\lie{sl}(m|n)$, $2^{m+n}m!n!$ for $\lie{osp}(2m+1|2n)$ and $2^{m+n-1}m!n!$ for $\lie{osp}(2m|2n)$.
			
			Moreover, by \cref{rootsordering} each of the above subgraphs again splits into several identical (up to a permutation of the edge colors) components without edges between them, corresponding to the possible reorderings of the simple roots.
			Observe that the ordering of the simple roots from \cref{slmnchevalley,oddospchevalley,evenospchevalley} is consistent under odd reflections.
			Therefore we only describe the component corresponding to this ordering, and for convenience also call it the Cartan graph.
			Since the number of simple roots is $m+n-1$ for $\lie{sl}(m|n)$ and $m+n$ for $\lie{osp}(2m+1|2n)$ and $\lie{osp}(2m|2n)$ it follows that altogether their Cartan graphs consist of $(m+n-1)!m!n!$ (resp.~$2^{m+n}m!n!(m+n)!$, $2^{m+n-1}n!m!(m+n)!$) copies of this component.
			By \cite[Thm.~3.1.3]{Musson} any two Borel subalgebras with the same even part are connected by a sequence of odd reflections, and therefore these components are moreover connected.
			
			In \cref{constructionBorelsl,constructionBoreloddosp,constructionBorelevenosp} we used partitions to describe the Borel subalgebras, and the corresponding shuffles to describe the simple roots.
			Also recall from \cref{oddreflectionpartition} that in this description odd reflections correspond to adding or removing single boxes.
			From these observations we obtain:
			\begin{prop}
				\label{cartangraphshape}
				The Cartan graphs of $\lie{sl}(m|n)$ (resp.~$\lie{gl}(n|n)$), $\lie{osp}(2m+1|2n)$ and $\lie{osp}(2m|2n)$ have the following underlying graphs:
				\begin{enumerate}
					\item For $\lie{sl}(m|n)$ (resp.~$\lie{gl}(n|n)$) the set of vertices is the set $\partitions_{m\times n}$ of partitions fitting in an $m\times n$-rectangle.
						The edges are colored by $\{1,\dots,m+n-1\}$.
						There are edges $\lambda\overset{i}{\longleftrightarrow}\lambda'$ if $\lambda'$ is obtained from $\lambda$ by adding a box numbered $i$ (using the numbering from \cref{oddreflectionpartition}), and loops $\lambda\overset{i}{\longleftrightarrow}\lambda$ if no box numbered $i$ can be added to $\lambda$.
					\item For $\lie{osp}(2m+1|2n)$  the set of vertices is $\partitions_{m\times n}$.
						The edges are colored by $\{1,\dots,m+n\}$.
						There are edges $\lambda\overset{i}{\longleftrightarrow}\lambda'$ if $\lambda'$ is obtained from $\lambda$ by adding a box numbered $i$, and loops $\lambda\overset{i}{\longleftrightarrow}\lambda$ if no box numbered $i$ can be added to $\lambda$.
						In particular there are loops of color $m+n$ at every vertex.
					\item For $\lie{osp}(2m|2n)$ the set of vertices is
						\begin{equation*}
							\{(\lambda,\eps)\mid\lambda\in\partitions_{m\times n}, \lambda_1<n, \eps\in\{+,-\}\}\cup\{(\lambda,\pm)\mid\lambda\in\partitions_{m\times n}, \lambda_1=n\}.
						\end{equation*}
						The edges are colored by $\{1,\dots,m+n\}$, and the non-loop edges are as follows:
						\begin{itemize}
							\item $(\lambda,\eps)\overset{i}{\longleftrightarrow}(\lambda',\eps)$ for $\lambda_1<n$ and $\lambda'$ obtained from $\lambda$ by adding a box numbered $i$, with $1\leq i\leq m+n-2$.
							\item $(\lambda,+)\overset{m+n-1}{\longleftrightarrow}(\lambda',\pm)$ for $\lambda_1=n-1$ and $\lambda'$ obtained from $\lambda$ by adding the box numbered $m+n-1$.
							\item $(\lambda,-)\overset{m+n}{\longleftrightarrow}(\lambda',\pm)$ for $\lambda_1=n-1$ and $\lambda'$ obtained from $\lambda$ by adding the box numbered $m+n-1$.
							\item $(\lambda,\pm)\overset{i}{\longleftrightarrow}(\lambda',\pm)$ for $\lambda_1=n$ and $\lambda'$ obtained from $\lambda$ by adding a box numbered $i$, with $1\leq i\leq m+n-2$.
						\end{itemize}
				\end{enumerate}
			\end{prop}
			Hence the connected components of the Cartan graph of $\lie{osp}(2m+1|2n)$ are almost the same as those of $\lie{sl}(m|n)$, with the only difference being the additional loops of color $m+n$ at every vertex for $\lie{osp}(2m+1|2n)$.
			For some concrete small examples see \cref{examplesection}
		\subsection{The Serre matrices}
			\begin{prop}\label{cartanmatrices}
				The Serre matrices for $\lie{sl}(m|n)$ (resp.~$\lie{gl}(n|n)$), $\lie{osp}(2m+1|2n)$ and $\lie{osp}(2m|2n)$ have the following form:
				\begin{enumerate}
					\item For $\lie{sl}(m|n)$ (resp.~$\lie{gl}(n|n)$) the Serre matrix is $A_{m+n-1}$ everywhere.
					\item For $\lie{osp}(2m+1|2n)$ the Serre matrix is $B_{m+n}$ everywhere.
					\item Let $\lambda\in\partitions_{m\times n}$ and $\eps\in\{+,-\}$.
						Then the Serre matrix at the vertex $(\lambda,\eps)$ of the Cartan graph of $\lie{osp}(2m|2n)$ is:
						\begin{itemize}
							\item $C_{m+n}$ if $\lambda_1<n-1$ and $\eps=+$,
							\item $C'_{m+n}$ (obtained by swapping the last two rows and columns of $C_{m+n}$) if $\lambda_1<n-1$ and $\eps=-$,
							\item $A_{m+n}$ if $\lambda_1=n-1$ and $\eps=+$,
							\item $A'_{m+n}$ (obtained by swapping the last two rows and columns of $A_{m+n}$) if $\lambda_1=n-1$ and $\eps=-$,
							\item $D_{m+n}$ if $\lambda_2=n$ and $\eps=\pm$,
							\item the generalized Cartan matrix
								\begin{equation*}
									\begin{smallpmatrix}
										2&-1&0&\cdots&0&0\\
										-1&2&\sddots&\sddots&\svdots&\svdots\\
										0&\sddots&\sddots&-1&0&0\\
										\svdots&\sddots&-1&2&-1&-1\\
										0&\cdots&0&-1&2&-1\\
										0&\cdots&0&-1&-1&2
									\end{smallpmatrix}
								\end{equation*}
								if $\lambda_2<\lambda_1=n$ and $\eps=\pm$.
						\end{itemize}
				\end{enumerate}
			\end{prop}
			\begin{proof}
				In all cases the Serre matrix is obtained from the Cartan data determined in \cref{slmnchevalley,oddospchevalley,evenospchevalley} according to the rules from \cref{associatedCMdef}.
			\end{proof}
			\begin{rem}
				Let $\lie{g}$ be $\lie{sl}(m|n)$ (resp.~$\lie{gl}(n|n)$), $\lie{osp}(2m+1|2n)$ or $\lie{osp}(2m|2n)$ and let $\lie{b}\subseteq\lie{g}$ be the Borel subalgebra corresponding to a vertex $x$ of the Cartan graph $\cartangr_\lie{g}$.
				Observe that the Serre matrix at $x$ is the generalized Cartan matrix corresponding to the Dynkin--Kac diagram for $\lie{b}$ considered as a Dynkin diagram, see \cite[\S 3.4.3]{Musson} for a list.
			\end{rem}
		\subsection{The Coxeter relations}
			By \cite[Thm.~9.4.8]{HeckenbergerSchneider} the Weyl groupoid of a Cartan graph $\cartangr=(I,X,r,A)$ is a \emph{Coxeter groupoid}, i.e.~the generators $s_i$ are only subject to relations of the form $\id_x(s_is_j)^{m(x)_{ij}}\id_x=\id_x$ for some symmetric matrices $(m(x)_{ij})$ with $m(x)_{ii}=1$ (with $i,j\in I$ and $x\in X$, and the implicit assumption that $(r_ir_j)^{m(x)_{ij}}(x)=x$ unless $m(x)_{ij}=\infty$).
			In fact $m(x)_{ij}=\abs{\cartanroots^\real_x\cap(\nats_0\alpha_i^x+\nats_0\alpha_j^x)}$.
			Therefore to obtain a presentation in terms of generators and relations we only have to determine the orders of $s_is_j$, starting from all vertices of the Cartan graph.
			\begin{prop}
				\label{coxeterrelations}
				For $\lie{sl}(m|n)$ (resp.~$\lie{gl}(n|n)$), $\lie{osp}(2m+1|2n)$ and $\lie{osp}(2m|2n)$, the $m(x)_{ij}$ are determined from the Serre matrices by the same rules as for semisimple Lie algebras.
				Explicitly,
				\begin{align*}
					A(x)_{ij}A(x)_{ji}=0&\implies m(x)_{ij}=2,\\
					A(x)_{ij}A(x)_{ji}=1&\implies m(x)_{ij}=3,\\
					A(x)_{ij}A(x)_{ji}=2&\implies m(x)_{ij}=4.
				\end{align*}
			\end{prop}
			\begin{proof}
				By \cref{cartanmatrices} we know for $\lie{sl}(m|n)$ (resp.~$\lie{gl}(n|n)$) and $\lie{osp}(2m+1|2n)$ that we have the same Serre matrices at every vertex in our Cartan graph.
				As the Serre matrices are the same at all vertices, the linear maps $s_i\colon \lie{h}^*\to\lie{h}^*$ corresponding to the generators of the Weyl groupoid are independent of the vertex.
				From this it follows that $(s_is_j)^{m(x)_{ij}}=\id_{\lie{h}^*}$, and that $m(x)_{ij}$ is the lowest number fulfilling this.
				Thus we only need to check that the induced path in the Cartan graph ends at the same vertex that we started.
				But this follows easily from the explicit description in terms of partitions.

				For $\lie{osp}(2m|2n)$ the situation is a bit more tedious.
				We need to compute the intersection of the linear span of two simple roots with the roots of $\lie{osp}(2m|2n)$.
				Using the explicit description of the simple roots in \cref{evenospchevalley} this is rather straightforward.
				We will only do this for the interesting cases, i.e.\  when $i,j\in\{n+m-2,n+m-1,n+m\}$, as the remaining cases are similar (and easier).
				\begin{itemize}
					\item For $x=(\lambda,+)$ with $\lambda_1\leq m-2$, the last three simple roots are $\eps_{i}-\eps_{m+n-1}$, $\eps_{m+n-1}-\eps_{m+n}$, $2\eps_{m+n}$, where $i$ is either $m$ or $m+n-2$.
						Therefore we get the claimed $m(x)_{ij}$.
					\item For $x=(\lambda,+)$ with $\lambda_1=m-1$, the last three simple roots are $\eps_{i}-\eps_{m}$, $\eps_{m}-\eps_{m+n}$, $2\eps_{m+n}$ where $i$ is either $m-1$ or $m+n-1$.
						As $2\eps_m$ is not a root of $\lie{osp}(2m|2n)$ we get type $A$ relations.
					\item For $x=(\lambda,-)$ we can apply the same argument as above since in this case only the last two simple roots are swapped.
					\item For $x=(\lambda,\pm)$ with $\lambda_1=m>\lambda_2$, the corresponding last three simple roots are given by $\eps_{i}-\eps_{m+n},\eps_{m+n}-\eps_{m},\eps_{m+n}+\eps_m$ where $i$ is either $m-1$ or $m+n-1$.
						As $2\eps_{m+n}$ is a root we directly see that $m(x)_{ij}=3$.
					\item Lastly suppose that $x=(\lambda, \pm)$ with $\lambda_1=\lambda_2=m$.
						The simple roots are then given by $\eps_{i}-\eps_{m-1}$, $\eps_{m-1}-\eps_{m}$, $\eps_{m-1}+\eps_m$ where $i$ is either $m+n$ or $m-2$.
						Now $2\eps_{m-1}$ is not a root of $\lie{osp}(2m|2n)$, therefore the linear span of $\eps_{m-1}-\eps_{m}$ and $\eps_{m-1}+\eps_m$ consists only of $2$ roots.
						Additionally $\eps_{i}\pm\eps_{m}$ are indeed roots, so we see the claimed type $D$ phenomenon.\qedhere
				\end{itemize}				
			\end{proof}
		\subsection{Automorphisms}
			For $\lie{sl}(m|n)$ (resp.~$\lie{gl}(n|n)$), $\lie{osp}(2m+1|2n)$ and $\lie{osp}(2m|2n)$ \cref{weylgroupoidautos} yields the following explicit description of the automorphism group of an object of the Weyl groupoid.
			\begin{cor}
				If $\lie{g}$ is $\lie{sl}(m|n)$(resp.~$\lie{gl}(n|n)$), $\lie{osp}(2m+1|2n)$ or $\lie{osp}(2m|2n)$ and $x$ any vertex of the Cartan graph of $\lie{g}$, then $\Aut_\weylgrpd(x)$ is isomorphic to the Weyl group of $\lie{g}_\even$.
			\end{cor}
			\begin{proof}
				We only have to show that any even root is principal, which follows easily from the explicit description in \cref{slmnchevalley,oddospchevalley,evenospchevalley}.
			\end{proof}
		\subsection{Some small examples}\label{examplesection}
			We explicitly describe the Weyl groupoids of $\lie{sl}(m|n)$, $\lie{osp}(2m+1|2n)$ and $\lie{osp}(2m|2n)$ in a few examples.
			\begin{ex}\label{sl22example}
					The contragredient Lie superalgebra $\lie{gl}(2|2)$ has three pairs of Chevalley generators, so the associated Weyl groupoid has three generators $s_1$, $s_2$ and $s_3$.
					By \cref{constructionBorelsl} we can index the Borel subalgebras (with a fixed even part) by partitions fitting into an $2\times 2$-rectangle, and by \cref{cartanmatrices} the Serre matrix is $\begin{smallpmatrix}2&-1&0\\-1&2&-1\\0&-1&2\end{smallpmatrix}$ everywhere.
					From the description of odd reflections in \cref{oddreflectionpartition} it follows that the generators of $\weylgrpd$ can be drawn in a diagram
					\begin{equation*}
						\begin{tikzcd}
							&&\ydiagram{2}\arrow[dr,"1", <->]\arrow[loop,in=125,out=55,looseness=4, <->, "2"']\\
							\emptyset\arrow[r, "2",<->]\arrow[loop,in=165,out=95,looseness=4, <->, "1"']\arrow[loop,in=195,out=265,looseness=4, <->, "3"]&\ydiagram{1}\arrow[ur, "3", <->]\arrow[dr, "1", <->]&&\ydiagram{1,2}\arrow[r, "2", <->]&\ydiagram{2,2}\arrow[loop,in=345,out=275,looseness=4, <->, "3"']\arrow[loop,in=15,out=75,looseness=4, <->, "1"]\\
							&&\ydiagram{1,1}\arrow[ur,"3", <->]\arrow[loop,in=235,out=305,looseness=4, <->, "2"]
						\end{tikzcd}
					\end{equation*}
					By \cref{coxeterrelations} the only relations are the familiar braid relations $s_1s_2s_1=s_2s_1s_2$, $s_2s_3s_2=s_3s_2s_3$ and $s_1s_3=s_3s_1$ (for all vertices of the Cartan graph).
			\end{ex}
			\begin{ex}\label{osp54example}
				For the Lie superalgebra $\lie{osp}(5|4)$, the underlying graph looks as in \cref{sl22example} but its Cartan subalgebra is $4$-dimensional instead of $3$-dimensional.
				So every vertex gets an additional loop with index $4$.
				
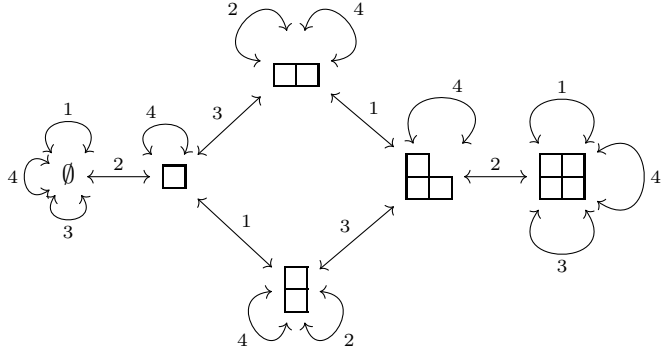
\begin{equation*}
					\begin{tikzcd}
						&&\ydiagram{0,2}\arrow[dr,"1", <->]\arrow[loop,in=175,out=105,looseness=4, <->, "2"']\arrow[loop,in=75,out=5,looseness=4, <->, "4"']\\
						\emptyset\arrow[r, "2",<->]\arrow[loop,in=125,out=55,looseness=4, <->, "1"']\arrow[loop,in=305,out=235,looseness=4, <->, "3"']\arrow[loop,in=215,out=145,looseness=4, <->, "4"']&\ydiagram{1}\arrow[loop, in=135,out=65,looseness=4, <->, "4"']\arrow[ur, "3", <->]\arrow[dr, "1", <->]&&\ydiagram{1,2}\arrow[loop, in=115,out=45,looseness=4, <->, "4"']\arrow[r, <->, "2"]&\ydiagram{2,2}\arrow[loop,in=125,out=55,looseness=4, <->, "1"']\arrow[loop,in=305,out=235,looseness=4, <->, "3"']\arrow[loop,in=35,out=325,looseness=4, <->, "4"']\\
						&&\ydiagram{1,1}\arrow[ur,"3", <->]\arrow[loop,in=355,out=285,looseness=4, <->, "2"']\arrow[loop,in=255,out=185,looseness=4, <->, "4"']
					\end{tikzcd}
				\end{equation*}
				In this case the Serre matrix is $B_4=\begin{smallpmatrix}2&-1&0&0\\-1&2&-1&0\\0&-1&2&-2\\0&0&-1&2\end{smallpmatrix}$ everywhere.
				The generators are subject to the usual ``type $B$ braid relations'', which are the relations from \cref{sl22example} as well as $s_3s_4s_3s_4=s_4s_3s_4s_3$ and $s_is_4=s_4s_i$ for $i\in\{1,2\}$.
			\end{ex}
			\Cref{sl22example} and \cref{osp54example} had in common that the Serre matrices were all the same at every vertex.
			This is however not true for $\lie{osp}(2m|2n)$:
			\begin{ex}\label{osp44example}
				The Cartan graph of $\lie{osp}(4|4)$ is
				
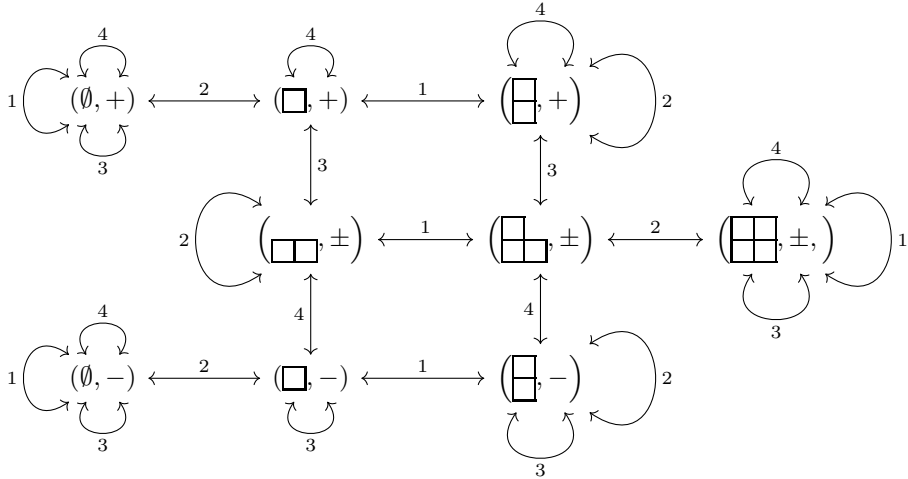
\begin{equation*}
					\begin{tikzcd}[row sep = large, column sep = large]
						\left(\emptyset,+\right)\arrow[r,<->,"2"]\arrow[loop,in=125,out=55,looseness=4, <->, "4"']\arrow[loop,in=305,out=235,looseness=4, <->, "3"']\arrow[loop,in=215,out=145,looseness=4, <->, "1"']&\left(\ydiagram{1},+\right)\arrow[loop,in=125,out=55,looseness=4, <->, "4"']\arrow[r, <->, "1"]\arrow[d,<->,"3"]&\left(\ydiagram{1,1},+\right)\arrow[loop,in=125,out=55,looseness=4, <->, "4"']\arrow[d,<->,"3"]\arrow[loop,in=35,out=325,looseness=4, <->, "2"']\\
						&\left(\ydiagram{0,2},\pm\right)\arrow[loop,in=215,out=145,looseness=4, <->, "2"']\arrow[r, <->, "1"]&\left(\ydiagram{1,2},\pm\right)\arrow[r, <->, "2"]&\left(\ydiagram{2,2},\pm,\right)\arrow[loop,in=125,out=55,looseness=4, <->, "4"']\arrow[loop,in=305,out=235,looseness=4, <->, "3"']\arrow[loop,in=35,out=325,looseness=4, <->, "1"']\\
						\left(\emptyset,-\right)\arrow[r,<->,"2"]\arrow[loop,in=125,out=55,looseness=4, <->, "4"']\arrow[loop,in=305,out=235,looseness=4, <->, "3"']\arrow[loop,in=215,out=145,looseness=4, <->, "1"']&\left(\ydiagram{1},-\right)\arrow[loop,in=305,out=235,looseness=4, <->, "3"']\arrow[r, <->, "1"]\arrow[u,<->,"4"]&\left(\ydiagram{1,1},-\right)\arrow[loop,in=305,out=235,looseness=4, <->, "3"']\arrow[u,<->,"4"]\arrow[loop,in=35,out=325,looseness=4, <->, "2"']
					\end{tikzcd}
				\end{equation*}
				The Serre matrix in the top left corner is of type $C_4=\begin{smallpmatrix}2&-1&0&0\\-1&2&-1&0\\0&-1&2&-1\\0&0&-2&2\end{smallpmatrix}$.
				The other two Serre matrices in the first row are of type $A_4=\begin{smallpmatrix}2&-1&0&0\\-1&2&-1&0\\0&-1&2&-1\\0&0&-1&2\end{smallpmatrix}$.
				The bottom row has the same Serre matrices as the first row except we swap the third and fourth row and column, i.e.~we have $\begin{smallpmatrix}2&-1&0&0\\-1&2&0&-1\\0&0&2&-2\\0&-1&-1&2\end{smallpmatrix}$ in the bottom left corner and $\begin{smallpmatrix}2&-1&0&0\\-1&2&0&-1\\0&0&2&-1\\0&-1&-1&2\end{smallpmatrix}$ for the other two.
				At $\left(\ydiagram{2,2}, \pm\right)$ we have $D_4=\begin{smallpmatrix}2&-1&0&0\\-1&2&-1&-1\\0&-1&2&0\\0&-1&0&2\end{smallpmatrix}$.
				The remaining two Serre matrices are given by $\begin{smallpmatrix}2&-1&0&0\\-1&2&-1&-1\\0&-1&2&-1\\0&-1&-1&2\end{smallpmatrix}$, in particular these are not of Dynkin type. 
				
				The generators of $\weylgrpd$ are subject to the braid relations (including relations of type $C=B$) specified by the Serre matrices.
			\end{ex}
	\appendix
	\section{Computation of Cartan data}\label{cartancomputation}
		In this appendix we explicitly describe the simple roots and Cartan data for the Borel subalgebras of $\lie{sl}(m|n)$, $\lie{osp}(2m+1|2n)$ and $\lie{osp}(2m|2n)$.
		For this it is more convenient to work with permutations instead of partitions, so we first need to set up a bit of combinatorics to pass between the two notions.
		\subsection{Combinatorics: Shuffles and partitions}
			Let $\partitions_{m\times n}$ be the set of partitions whose Young diagram fits into an $m\times n$-rectangle.
			We use the slightly unusual convention that the longest row is at the bottom, so for instance the diagram \ydiagram{1,2,4} represents the partition $\lambda=(4,2,1)$.
			
			Recall that a permutation $\sigma\in S_{m+n}$ is an \emph{$(m,n)$-shuffle} if $\sigma^{-1}(i)<\sigma^{-1}(j)$ for all pairs $i<j$ with either $i,j\leq m$ or $i,j>m$.
			We write $\shuffles(m,n)$ for the set of $(m,n)$-shuffles.
			Equivalently, $\shuffles(m,n)$ can be defined as a set of shortest coset representatives for the parabolic quotient $(S_m\times S_n)\backslash S_{m+n}$.
			Note that if $\sigma$ is an $(m,n)$-shuffle, then either $\sigma(m+n)=m$ or $\sigma(m+n)=m+n$.
			
			We will identify shuffles with partitions as follows:
			\begin{lemma}
				\label{shufflespartitions}
				There is a bijection between $\shuffles(m,n)$ and the set of Young diagrams (partitions) fitting into an $m\times n$-rectangle, as follows:
				for an $(m,n)$-shuffle $\sigma$ we draw a path in the $m\times n$-rectangle, where in the $i$-th step we go down if $\sigma(i)\leq m$ and right if $\sigma(i)>m$.
				Then the partition $\lambda$ consists of the boxes below the path.
			\end{lemma}
			This bijection is best explained in an example.
			\begin{ex}
				Let $m=3$, $n=4$ and $\sigma=\begin{smallpmatrix}1&2&3&4&5&6&7\\4&1&5&2&6&7&3\end{smallpmatrix}\in\shuffles(3,4)$.
				According to \cref{shufflespartitions} $\sigma$ encodes the boundary path $rdrdrrd$ (where $r$ means ``right'' and $d$ ``down''):
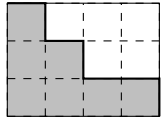
				\begin{equation}
					\label{partitioninrectangle}
					\begin{tikzpicture}
						\fill[lightgray] (0,0) -- (0.5,0) -- (0.5,-0.5) -- (1,-0.5) -- (1,-1) -- (2,-1) -- (2,-1.5) -- (0,-1.5) -- (0,0);
							
						\draw[step=0.5,dashed] (0,0) grid (2,-1.5);
						\draw (0,0) rectangle (2,-1.5);
						
						\draw[thick] (0,0) -- (0.5,0) -- (0.5,-0.5) -- (1,-0.5) -- (1,-1) -- (2,-1) -- (2,-1.5);
					\end{tikzpicture}
				\end{equation}
				The permutation corresponding to $\sigma$ is $\lambda=(4,2,1)$.
			\end{ex}
		\subsection{Cartan data for \texorpdfstring{$\lie{sl}(m|n)$}{sl(m|n)}}
			\begin{prop}
				\label{slmnchevalley}
				Let $\sigma\in\shuffles(m,n)$.
				The simple roots corresponding to the Borel subalgebra $\lie{b}(\sigma)$ of $\lie{sl}(m|n)$ are
				\begin{equation*}
					\Pi(\sigma)=\{\alpha_i=\eps_{\sigma(i)}-\eps_{\sigma(i+1)}\mid 1\leq i\leq m+n-1\}.
				\end{equation*}
				For the corresponding Cartan datum $(B,\tau)$ we have $\tau_i=\even$ if either $\sigma(i),\sigma(i+1)\leq m$ or $\sigma(i),\sigma(i+1)>m$, and $\tau_i=\odd$ otherwise.
				The $i$-th row of the matrix $B$ is given by
				\begin{equation*}
					(b_{i,1},\dots,b_{i,m+n-1})=\begin{cases}
						(0,\dots,0,-1,2,-1,0,\dots,0)&\text{if }|e_i|=\even,\\
						(0,\dots,0,-1,0,1,0,\dots,0)&\text{if }|e_i|=\odd,
					\end{cases}
				\end{equation*}
				where the entry $2$ (resp.~the ``middle'' $0$) is in the $i$-th spot.
			\end{prop}
			\begin{proof}
				The simple roots are listed in \cite[Lem.~3.4.3]{Musson}.
				Since the elementary matrix $E_{rs}$ is of weight $\eps_r-\eps_s$ we can take
				\begin{align*}
					e_i&=E_{\sigma(i),\sigma(i+1)},&
					f_i&=E_{\sigma(i+1),\sigma(i)}
				\end{align*}
				as Chevalley generators.
				Clearly $e_i$ (and $f_i$) is even if and only if $\sigma(i)$ and $\sigma(i+1)$ are either both $\leq m$ or both $\geq m+1$.
				Therefore $h_i=[e_i,f_i]=E_{\sigma(i),\sigma(i)}-(-1)^{|e_i|}E_{\sigma(i+1),\sigma(i+1)}$, and thus
				\begin{equation*}
					\alpha_j(h_i)=\begin{cases}
						0&\text{if }j\neq i, i\pm 1,\\
						2&\text{if }j=i, |e_i|=\even,\\
						0&\text{if }j=i, |e_i|=\odd,\\
						-1&\text{if }j=i-1,\\
						-(-1)^{\tau_i}&\text{if }j=i+1.
					\end{cases}
				\end{equation*}
				This shows that the matrix $B$ has the claimed form.
			\end{proof}
		\subsection{Cartan data for \texorpdfstring{$\lie{osp}(2m+1|2n)$}{osp(2m+1|2n)}}
			\begin{prop}
				\label{oddospchevalley}
				Let $\sigma\in\shuffles(m,n)$.
				The simple roots for the Borel subalgebra $\lie{b}(\sigma)$ of $\lie{osp}(2m+1|2n)$ are
				\begin{equation*}
					\Pi(\sigma)=\{\alpha_i=\eps_{\sigma(i)}-\eps_{\sigma(i+1)}\mid 1\leq i\leq m+n-1\}\cup\{\alpha_{m+n}=\eps_{\sigma(m+n)}\}.
				\end{equation*}
				The corresponding Cartan datum $(B,\tau)$ can be described as follows.
				For $1\leq i\leq m+n-1$ we have $\tau_i=\even$ if either $\sigma(i),\sigma(i+1)\leq m$ or $\sigma(i),\sigma(i+1)>m$, and $\tau_{m+n}=\even$ if $\sigma(m+n)\leq m$ and $\tau_{m+n}=\odd$ if $\sigma(m+n)>m$.
				The $i$-th row of the matrix $B$ is given by
				\begin{equation*}
					(b_{i,1},\dots,b_{i,m+n-1})=\begin{cases}
						(0,\dots,0,-1,2,-1,0,\dots,0)&\text{if }i<m+n, |e_i|=\even,\\
						(0,\dots,0,-1,0,1,0,\dots,0)&\text{if }i<m+n, |e_i|=\odd,\\
						(0,\dots,0,-1,1)&\text{if }i=m+n,
					\end{cases}
				\end{equation*}
				where the entry $2$ (resp.~the ``middle'' $0$) is in the $i$-th spot.
			\end{prop}
			\begin{proof}
				The simple roots are listed in \cite[Lem.~3.4.3]{Musson}.
				From the explicit description of the root spaces (see e.g.~\cite[Exercise~2.7.4]{Musson}) it follows that one possible choice for the Chevalley generators (for $1\leq i\leq m+n-1$) is
				\begin{align*}
					e_i&=\begin{cases}
						E_{\sigma(i),\sigma(i+1)}+E_{-\sigma(i+1),-\sigma(i)}&\text{if }\sigma(i)\leq m,\sigma(i+1)>m,\\
						E_{\sigma(i),\sigma(i+1)}-E_{-\sigma(i+1),-\sigma(i)}&\text{otherwise},
					\end{cases}\\
					e_{m+n}&=E_{\sigma(m+n),0}-E_{0,-\sigma(m+n)}\\
					f_i&=\begin{cases}
						E_{\sigma(i+1),\sigma(i)}+E_{-\sigma(i),-\sigma(i+1)}&\text{if }\sigma(i)>m,\sigma(i+1)\leq m,\\
						E_{\sigma(i+1),\sigma(i)}-E_{-\sigma(i),-\sigma(i+1)}&\text{otherwise},
					\end{cases}\\
					f_{m+n}&=\begin{cases}
						E_{0,m}-E_{-m,0}&\text{if }\sigma(m+n)=m,\\
						E_{0,m+n}+E_{-(m+n),0}&\text{if }\sigma(m+n)=m+n.
					\end{cases}
				\end{align*}
				Hence $e_i$ (and thus $f_i$) is even if and only if $\sigma(i)$ and $\sigma(i+1)$ are either both $\leq m$ or both $\geq m+1$.
				From this it follows that
				\begin{equation*}
					h_i=[e_i,f_i]=E_{\sigma(i),\sigma(i)}-E_{-\sigma(i),-\sigma(i)}-(-1)^{|e_i|}(E_{\sigma(i+1),\sigma(i+1)}-E_{-\sigma(i+1),-\sigma(i+1)})
				\end{equation*}
				for $1\leq i\leq m+n-1$, and (independent of the value of $\sigma(m+n)$)
				\begin{equation*}
					h_{m+n}=E_{\sigma(m+n),\sigma(m+n)}-E_{-\sigma(m+n),-\sigma(m+n)}.
				\end{equation*}
				Therefore the entries of the matrix $B$ are
				\begin{equation*}
					b_{ij}=\alpha_j(h_i)=\begin{cases}
						0&\text{if }j\neq i,i\pm 1,\\
						2&\text{if }j=i<m+n, \tau_i=\even,\\
						0&\text{if }j=i<m+n, \tau_i=\odd,\\
						1&\text{if }j=i=m+n,\\
						-1&\text{if }j=i-1,\\
						-(-1)^{\tau_i}&\text{if }j=i+1,
					\end{cases}
				\end{equation*}
				and thus $B$ has the claimed form.
			\end{proof}
		\subsection{Cartan data for \texorpdfstring{$\lie{osp}(2m|2n)$}{osp(2m|2n)}}
			\begin{prop}
				\label{evenospchevalley}
				Let $\sigma\in\shuffles(m,n)$ and $\eps\in\{+,-\}$.
				The simple roots and the Cartan data $(B,\tau)$ for the Borel subalgebras $\lie{b}(\sigma,\eps)$ of $\lie{osp}(2m|2n)$ can be described as follows:
				\begin{enumerate}
					\item For $\sigma(m+n)=m+n$ and $\lie{b}(\sigma,+)$ the simple roots are
						\begin{equation*}
							\Pi(\sigma,+)=\{\alpha_i=\eps_{\sigma(i)-\sigma(i+1)}\mid 1\leq i\leq m+n-1\}\cup\{\alpha_{m+n}=2\eps_{m+n}\}.
						\end{equation*}
						The Chevalley generators $e_i$ and $f_i$ for $1\leq i\leq m+n-1$ are even if and only if either $\sigma(i),\sigma(i+1)\leq m$ or $\sigma(i),\sigma(i+1)>m$, while $e_{m+n}$ and $f_{m+n}$ are even.
						The $i$-th row of $B$ is
						\begin{equation*}
							(b_{i,1},\dots,b_{i,m+n})=
							\begin{cases}
								(0,\dots,0,-1,2,-1,0,\dots,0)&\text{if }i\neq m+n-1, |e_i|=\even,\\
								(0,\dots,0,-1,0,1,0,\dots,0)&\text{if }i\neq m+n-1, |e_i|=\odd,\\
								(0,\dots,0,-1,2,-2)&\text{if }i=m+n-1, |e_i|=\even,\\
								(0,\dots,0,-1,0,2)&\text{if }i=m+n-1, |e_i|=\odd.
							\end{cases}
						\end{equation*}
						again with the entry $2$ (resp.~the ``middle'' $0$) in the $i$-th spot.
					\item For $\sigma(m+n)=m+n$ and $\lie{b}(\sigma,-)$ the simple roots $\Pi(\sigma,-)$ are obtained from  $\Pi(\sigma,+)$ by replacing every $\eps_m$ by $\eps_{-m}=-\eps_m$, and swapping the $m+n-1$-th and the $m+n$-th simple root.
						The Cartan datum is as above except that the last two rows and columns of the matrix $B$ are swapped.
					\item For $\sigma(m+n)=m$ and $\lie{b}(\sigma,\pm)$ the simple roots are
						\begin{multline*}
							\Pi(\sigma,\pm)=\{\alpha_i=\eps_{\sigma(i)-\sigma(i+1)}\mid 1\leq i\leq m+n-1\}\\
							\cup\{\alpha_{m+n}=\eps_{\sigma(m+n-1)}+\eps_{\sigma(m+n)}\}.
						\end{multline*}
						The Chevalley generators $e_i$ and $f_i$ for $1\leq i\leq m+n-1$ are even if and only if either $\sigma(i),\sigma(i+1)\leq m$ or $\sigma(i),\sigma(i+1)>m$, while $e_{m+n}$ and $f_{m+n}$ have the same parity as $e_{m+n-1}$.
						The $i$-th row $(b_{i,1},\dots,b_{i,m+n})$ of the matrix $B$ is given by the following table:
						\begin{equation*}
							\begin{array}{c|c|c}
								&\tau_i=\even&\tau_i=\odd\\
								\hline
								i<m+n-2&(0,\dots,0,-1,2,-1,0,\dots,0)&(0,\dots, 0,-1,0,1,0,\dots,0)\\
								i=m+n-2&(0,\dots,0,-1,2,-1,-1)&(0,\dots,0,-1,0,1,1)\\
								i=m+n-1&(0,\dots,0,-1,2,0)&(0,\dots,0,-1,0,2)\\
								i=m+n&(0,\dots,0,-1,0,2)&(0,\dots,0,-1,2,0)
							\end{array}
						\end{equation*}
				\end{enumerate}
			\end{prop}
			\begin{proof}
				The simple roots are listed in \cite[Lem.~3.4.3]{Musson}.
				A choice for the Chevalley generators $e_i$, $f_i$ and $h_i=[e_i,f_i]$ associated with the simple roots is:
				\begin{itemize}
					\item for $\alpha_i=\eps_{\sigma(i)}-\eps_{\sigma(i+1)}$,
						\begin{align*}
							e_i&=\begin{cases}
								E_{\sigma(i),\sigma(i+1)}+E_{-\sigma(i+1),-\sigma(i)}&\text{if }\sigma(i)\leq m,\sigma(i+1)>m,\\
								E_{\sigma(i),\sigma(i+1)}-E_{-\sigma(i+1),-\sigma(i)}&\text{otherwise},
							\end{cases}\\
							f_i&=\begin{cases}
								E_{\sigma(i+1),\sigma(i)}+E_{-\sigma(i),-\sigma(i+1)}&\text{if }\sigma(i)>m,\sigma(i+1)\leq m,\\
								E_{\sigma(i+1),\sigma(i)}-E_{-\sigma(i),-\sigma(i+1)}&\text{otherwise},
							\end{cases}\\
							h_i&=E_{\sigma(i),\sigma(i)}-E_{-\sigma(i),-\sigma(i)}-(-1)^{|e_i|}(E_{\sigma(i+1),\sigma(i+1)}-E_{-\sigma(i+1),-\sigma(i+1)}).
						\end{align*}
					\item for $\alpha_i=\eps_{\sigma(i)}+\eps_m$,
						\begin{align*}
							e_i&=E_{\sigma(i),-m}-E_{m,-\sigma(i)},\\
							f_i&=\begin{cases}
								E_{-m,\sigma(i)}-E_{-\sigma(i),m}&\text{if }\sigma(i)\leq m,\\
								E_{-m,\sigma(i)}+E_{-\sigma(i),m}&\text{if }\sigma(i)>m,\\
							\end{cases}\\
							h_i&=E_{\sigma(i),\sigma(i)}-E_{-\sigma(i),-\sigma(i)}+(-1)^{|e_i|}(E_{m,m}-E_{-m,-m}).
						\end{align*}
					\item for $\alpha_i=-\eps_m-\eps_{\sigma(i+1)}$,
						\begin{align*}
							e_i&=\begin{cases}
								E_{-m,\sigma(i+1)}-E_{-\sigma(i+1),m}&\text{if }\sigma(i+1)\leq m,\\
								E_{-m,\sigma(i+1)}+E_{-\sigma(i+1),m}&\text{if }\sigma(i+1)>m,\\
							\end{cases}\\
							f_i&=E_{\sigma(i+1),-m}-E_{m,-\sigma(i+1)},\\
							h_i&=E_{-m,-m}-E_{m,m}-(-1)^{|e_i|}(E_{\sigma(i+1),\sigma(i+1)}-E_{-\sigma(i+1),-\sigma(i+1)}).
						\end{align*}
					\item for $\alpha_i=2\eps_{m+n}$,
						\begin{align*}
							e_i&=E_{m+n,-(m+n)},&
							f_i&=E_{-(m+n),m+n},\\
							h_i=[e_i,f_i]&=E_{m+n,m+n}-E_{-(m+n),-(m+n)}.
						\end{align*}
				\end{itemize}
				Now we plug the $h_i$ into the $\alpha_j$.
				For this we have to consider each of the three classes of Borel subalgebras separately.
				\begin{caselist}
					\item $\sigma(m+n)=m+n$, $+$.
						Then
						\begin{equation*}
							\alpha_j(h_i)=\begin{cases}
								0&\text{if }j\neq i,i\pm 1,\\
								2&\text{if }j=i,\tau_i=\even,\\
								0&\text{if }j=i,\tau_i=\odd,\\
								-1&\text{if }j=i-1,\\
								-(-1)^{\tau_i}&\text{if }j=i+1<m+n,\\
								-2(-1)^{\tau_i}&\text{if }j=i+1=m+n,
							\end{cases}
						\end{equation*}
						and it follows that the matrix $B$ has the claimed form.
					\item $\sigma(m+n)=m+n$, $-$.
						In this case let $i_0=\sigma^{-1}(m)$.
						The simple roots are as in the previous case except that $\eps_{\sigma(i_0-1)}-\eps_m$ and $\eps_m-\eps_{\sigma(i_0+1)}$ are replaced by $\eps_{\sigma(i_0-1)}+\eps_m$ and $-\eps_m-\eps_{\sigma(i_0+1)}$, respectively, and the numbering is changed.
						The values $\alpha_j(h_i)$ are as in the previous case except we have to check the cases involving $i_0$ separately.
						Nevertheless it follows that the matrix $B$ is as in Case~I except we have to swap the last two rows and columns to account for the renumbering of the simple roots.
					\item $\sigma(m+n)=m$.
						Similarly to the previous cases we obtain
						\begin{equation*}
							\alpha_j(h_i)=\begin{cases}
								0&\text{if }j\neq i, i\pm 1, i\pm 2,\\
								2&\text{if }j=i, \tau_i=\even,\\
								0&\text{if }j=i, \tau_i=\odd,\\
								-1&\text{if }j=i-1, i<m+n,\\
								1-(-1)^{\tau_i}&\text{if }j=i-1, i=m+n,\\
								0&\text{if }j=i-2, i<m+n,\\
								-1&\text{if }j=i-2, i=m+n,\\
								-(-1)^{\tau_i}&\text{if }j=i+1<m+n,\\
								1+(-1)^{\tau_i}&\text{if }j=i+1=m+n,\\
								0&\text{if }j=i+2<m+n,\\
								-(-1)^{\tau_i}&\text{if }j=i+2=m+n,
							\end{cases}
						\end{equation*}
						as claimed.\qedhere
				\end{caselist}
			\end{proof}
	\clearpage
	\printbibliography
\end{document}

%% file: WeylGroupoidspreamble.tex
\usepackage[english]{babel}

\usepackage[T1]{fontenc}
\usepackage{lmodern}

\usepackage{amsmath}

\usepackage{amsfonts}
\usepackage{amssymb}
\usepackage{mathtools}

\usepackage[backend=bibtex,style=alphabetic,giveninits=true,maxnames=99]{biblatex}
\usepackage{csquotes}

\usepackage{graphicx}
\usepackage{float}

\usepackage{caption}
\usepackage{subcaption}

\usepackage{enumitem}

\usepackage{tikz}
\usetikzlibrary{matrix}
\usetikzlibrary{calc}
\usetikzlibrary{positioning}
\usepackage{tikz-cd}

\usepackage[smalltableaux]{ytableau}

\usepackage[linktocpage]{hyperref}
\usepackage{amsthm}
\usepackage[capitalize]{cleveref}

\usepackage{ytableau}

\ytableausetup{centertableaux, smalltableaux}

\newcounter{thm}
\numberwithin{thm}{section}						
\numberwithin{equation}{section}				

\newtheorem{lemma}[thm]{Lemma}
\newtheorem{prop}[thm]{Proposition}
\newtheorem{cor}[thm]{Corollary}

\theoremstyle{definition}
\newtheorem{dfn}[thm]{Definition}
\newtheorem{comp}{Comparison}
\newtheorem{ex}[thm]{Example}

\theoremstyle{remark}
\newtheorem{rem}[thm]{Remark}

\crefname{lemma}{Lemma}{Lemmas}
\Crefname{lemma}{Lemma}{Lemmas}
\crefname{prop}{Proposition}{Propositions}
\Crefname{prop}{Proposition}{Propositions}
\crefname{cor}{Corollary}{Corollaries}
\Crefname{cor}{Corollary}{Corollaries}
\crefname{dfn}{Definition}{Definitions}
\Crefname{dfn}{Definition}{Definitions}
\crefname{ex}{Example}{Examples}
\Crefname{ex}{Example}{Examples}
\crefname{rem}{Remark}{Remarks}
\Crefname{rem}{Remark}{Remarks}
\crefname{equation}{}{}
\Crefname{equation}{}{}
\crefname{enumi}{}{}
\Crefname{enumi}{}{}

\setlist[itemize]{itemsep=0mm}
\setlist[itemize,2]{label=$\circ$}
\setlist[enumerate]{itemsep=0mm}
\setlist[description]{itemsep=0mm}
\setlist[enumerate,1]{label=\arabic*)}	

\newlist{caselist}{enumerate}{1}
\setlist[caselist]{label=Case~\Roman*:,align=left,itemindent=0pt,leftmargin=1.5\parindent,labelwidth=*,labelindent=0.5\parindent}

\newcommand{\compl}{\mathbb{C}}

\newcommand{\ints}{\mathbb{Z}}
\newcommand{\nats}{\mathbb{N}}

\newcommand{\even}{{\bar{0}}}
\newcommand{\odd}{{\bar{1}}}

\newcommand{\eps}{\varepsilon}					

\newenvironment{smallpmatrix}{\left(\begin{smallmatrix}}{\end{smallmatrix}\right)}

\makeatletter
\DeclareRobustCommand\svdots{%
  \mathpalette\@svdots{}%
}
\newcommand*{\@svdots}[2]{%
  \sbox0{$#1\cdotp\cdotp\cdotp\m@th$}%
  \sbox2{$#1.\m@th$}%
  \vbox{%
    \dimen@=\wd0 %
    \advance\dimen@ -3\ht2 %
    \kern.5\dimen@
    \dimen@=\wd2 %
    \advance\dimen@ -\ht2 %
    \dimen2=\wd0 %
    \advance\dimen2 -\dimen@
    \vbox to \dimen2{%
      \offinterlineskip
      \copy2 \vfill\copy2 \vfill\copy2 %
    }%
  }%
}
\DeclareRobustCommand\sddots{%
  \mathinner{%
    \mathpalette\@sddots{}%
    \mkern\thinmuskip
  }%
}
\newcommand*{\@sddots}[2]{%
  \sbox0{$#1\cdotp\cdotp\cdotp\m@th$}%
  \sbox2{$#1.\m@th$}%
  \vbox{%
    \dimen@=\wd0 %
    \advance\dimen@ -3\ht2 %
    \kern.5\dimen@
    \dimen@=\wd2 %
    \advance\dimen@ -\ht2 %
    \dimen2=\wd0 %
    \advance\dimen2 -\dimen@
    \vbox to \dimen2{%
      \offinterlineskip
      \hbox{$#1\mathpunct{.}\m@th$}%
      \vfill
      \hbox{$#1\mathpunct{\kern\wd2}\mathpunct{.}\m@th$}%
      \vfill
      \hbox{$#1\mathpunct{\kern\wd2}\mathpunct{\kern\wd2}\mathpunct{.}\m@th$}%
    }%
  }%
}
\makeatother

\newcommand{\lie}[1]{\mathfrak{#1}}				

\DeclareMathOperator{\rk}{rk}					

\DeclareMathOperator{\id}{id}					
\DeclareMathOperator{\ad}{ad}					

\newcommand{\GL}{\mathrm{GL}}					
\newcommand{\Aut}{\mathrm{Aut}}					

\newcommand{\cartangr}{\mathcal{G}}				
\newcommand{\weylgrpd}{\mathcal{W}}				

\newcommand{\cartanroots}{\Delta}				
\DeclareMathOperator{\tr}{tr}	                

\DeclarePairedDelimiter\abs{\lvert}{\rvert}     

\newcommand{\shuffles}{\mathrm{Shff}}			
\newcommand{\partitions}{\mathcal{P}}			

\newcommand{\GCM}{\mathrm{GCM}}					

\newcommand{\real}{\mathrm{real}}				

\pgfdeclarelayer{bg}
\pgfsetlayers{bg,main}